\documentclass[sn-mathphys-num]{sn-jnl}

\usepackage{graphicx}%
\usepackage{multirow}%
\usepackage{amsmath,amssymb,amsfonts}%
\usepackage{amsthm}%
\usepackage{mathrsfs}%
\usepackage[title]{appendix}%
\usepackage{xcolor}%
\usepackage{textcomp}%
\usepackage{manyfoot}%
\usepackage{booktabs}%
\usepackage{algorithm}%
\usepackage{algorithmicx}%
\usepackage{algpseudocode}%
\usepackage{listings}%

\usepackage{tcolorbox}

\usepackage{tikz} 
\usetikzlibrary{calc,arrows,intersections}
\usetikzlibrary{patterns}

\theoremstyle{thmstyleone}%
\newtheorem{theorem}{Theorem}[section]

\newtheorem{conjecture}[theorem]{Conjecture}
\newtheorem{corollary}[theorem]{Corollary}

\newtheorem{lemma}[theorem]{Lemma}

\newtheorem{proposition}[theorem]{Proposition}

\numberwithin{equation}{section}

\theoremstyle{thmstyletwo}%
\newtheorem{example}{Example}%
\newtheorem{remark}{Remark}%

\theoremstyle{thmstylethree}%
\newtheorem{definition}{Definition}%

\DeclareMathOperator{\conv}{conv}

\DeclareMathOperator{\EH}{EH}

\DeclareMathOperator{\argmin}{argmin}

\DeclareMathOperator{\rec}{rec}
\DeclareMathOperator{\CL}{CL}
\DeclareMathOperator{\FW}{FW}
\DeclareMathOperator{\SD}{SD}

\newcommand\RR{{\mathbb R}}

\newcommand\0{{\mathbf 0}}
\newcommand\1{{\mathbf 1}}

\newcommand\SetOf[2]{\left\{#1 : #2\right\}}

\newcommand\torus[1]{\RR^{#1}/\RR\1} 

\begin{document}

\title[Breakdown points of Fermat--Weber problems]{Breakdown points of Fermat--Weber problems under gauge distances}


\author*[1]{\fnm{Andrei} \sur{Com\u{a}neci}}\email{andrei.comaneci@yandex.com}

\author[2]{\fnm{Frank} \sur{Plastria}}\email{Frank.Plastria@vub.be}

\affil*[1]{\orgdiv{Institut f\"{u}r Mathematik}, \orgname{Technische Universit\"{a}t Berlin}, \orgaddress{\street{Str.~des~17.~Juni 136}, \city{Berlin}, \postcode{10587}, \country{Germany}}}

\affil[2]{\orgdiv{Business Technology and Operations}, \orgname{Vrije Universiteit Brussel}, \orgaddress{\street{Pleinlaan~2}, \city{Brussel}, \postcode{1050}, \country{Belgium}}}

\abstract{
We compute the robustness of Fermat--Weber points with respect to any  finite gauge.
We show a breakdown point of $1/(1+\sigma)$ where $\sigma$ is the asymmetry measure of the gauge.
We obtain quantitative results indicating how far a corrupted Fermat--Weber point can lie from the true value in terms of the original sample and the size of the corrupted part.
If the distance from the true value depends only on the original sample, then we call the gauge `uniformly robust.'
We show that polyhedral gauges are uniformly robust, but locally strictly convex norms are not, while in dimension 2 any uniform robust gauge is polyhedral.
}




\maketitle

\section{Introduction} \label{sec:intro}

The concept of robustness in statistical analysis originated from the need to address uncertainties in real-world data.
Data sets often contain errors, either due to faulty observations, transcription mistakes, or even deliberate interference.
The fundamental concern is that such imperfections should not excessively skew or distort the conclusions drawn from the data.
Robust statistics aim to develop methods that remain reliable and stable in the presence of such anomalies.
For an overview of these ideas, refer to the books of Huber~\cite{Huber1981} and Staudte \& Sheather~\cite{Staudte+Sheather:1990}.

One of the most intuitive examples of robustness in practice is comparing the behavior of different estimators of central tendency, such as the median and the mean.
The mean is highly sensitive to outliers—values that deviate significantly from the general trend of the data.
A single extreme observation can shift the mean drastically, leading to potentially misleading conclusions.
On the other hand, the median, as a more robust estimator, tends to resist such distortions.
In fact, even if nearly half the data is corrupted by outliers, the median remains largely unaffected, making it a preferable choice in many cases where data integrity is questionable.

Robustness is quantitatively evaluated using concepts such as the \emph{breakdown point}, which defines the proportion of contaminated data an estimator can handle before yielding arbitrarily large errors.
The higher the breakdown point, the more robust the estimator.
For instance, the median has the maximal breakdown point of 0.5, whereas the mean breaks down with a single extreme value.

While the breakdown point is a key criterion in assessing robustness, alternative robustness measures exist, particularly for location problems; see \cite{Carrizosa+Nickel:2003}.
These measures often focus on small errors or perturbations rather than the presence of outliers, so we will not further discuss such cases.

A common method for creating robust estimators is to solve an optimization problem of the form
\begin{equation} \label{eq:M-estim}
    \text{minimize}_{x\in\RR^d}\quad\sum_{a\in A}\rho(x-a),
\end{equation}
where $A$ is the set of data and $\rho$ is a function that measures dissimilarity. The optimal solutions to \eqref{eq:M-estim} are called M-estimators.

In dimension $d=1$ the mean, which is not robust, arises when $\rho(x) = |x|^2$.
To obtain robustness, one must impose further restrictions on $\rho$.

The median, which results from using $\rho(x) = |x|$, is robust because the defining function has linear asymptotic growth, as opposed to the quadratic growth associated with the mean.
Quadratic growth causes extreme outliers to disproportionately influence the displacement of the estimator.
In contrast, the linear growth ensures that gross outliers do not affect the objective significantly more than other points.

Combining the properties of the mean and median, Huber's estimator
with parameter $\delta>0$, obtained with $\rho(x)=\frac{1}{2}x^2$ for $|x|\leq \delta$
and $\rho(x)=\delta|x|-\frac{1}{2}\delta^2$ otherwise, is also robust \cite{Huber1981}. 

Although M-estimators are well-known and extensively studied, the primary focus has been on the symmetric case, where $\rho(-x) = \rho(x)$ for all $x \in \RR^d$.
Most studies assume symmetric underlying distributions, likely due to the desirable properties these estimators exhibit~\cite{Freedman+Diaconis:1982,Huber1981,Mizera:1994}.

However, real-world data is often skewed, so a symmetric estimator can increase bias in the presence of outliers.
In such cases, estimators based on quantiles, which are also M-estimators derived from asymmetric functions~$\rho$, are sometimes preferred.
Asymmetric M-estimators (AM-estimators) offer advantages for certain types of distributions.
For example, \cite{AFGP:2006} demonstrates the benefits of AM-estimators in image processing, \cite{Wang+Lee:2011} applies them to Burr Type III distributions, and \cite{DCC:2018} uses them in microscopy.
Additionally, \cite{Xu+Chen:2018} introduced an asymmetric variant of Huber’s function and tested their method on failure times of steel specimens in fatigue tests.
Our study aims to initiate a better theoretical understanding of the robustness of similar estimators in a multidimensional setting.

Asymmetry of distance measures is well recognized in location theory
\cite{Plastria:2009}, particularly for physical distance when moving through a network with one-way streets \cite{Witzgall:TechnicalReport:1964}, 
for transportation effort on an inclined terrain \cite{Hodgson1987}, 
for time-distance in rush-hour road travel, flight in the presence of wind \cite{Drezner1989}, navigation in the presence of current, interception of a moving target \cite{Cera2007}, 
etc.

When generalizing to multidimensional data, the median corresponds to the Fermat--Weber point in location theory.
This point minimizes the sum of weighted Euclidean distances to a given sample in $\RR^d$.
The robustness of the Fermat–Weber point has been studied extensively, with \cite{bdpoint} demonstrating that its Euclidean version retains a breakdown point of 0.5, indicating it can tolerate corruption in up to half of the data without being significantly affected.

In \cite{Plastria2016}, the second author mistakenly claimed that the same robustness property applies to the Fermat–Weber problem with potentially asymmetric distance measures.
However, as the first author later pointed out, this assertion is incorrect.
Any one-dimensional $b$-quantile can be viewed as a Fermat--Weber point for some suitable distance measure, but its breakdown point is $\min(b,1-b)$; see \cite[\textsection 3.2.1]{Staudte+Sheather:1990} or \cite[Lemma~4.2]{Duembge:2016}.
Upon further review, it was found that the arguments in \cite{Plastria2016} hold only when the distance is derived from a norm, meaning it must be symmetric; see \cite{Plastria2023}.

Interestingly, a generalization of quantiles to the multivariate case was considered in \cite{Koltchinskii-M-quantiles:1997} and \cite{Chaudhuri-spatial-quantile:1996} which gave rise to the notion of $M$-quantiles---they are Fermat--Weber points with respect to skewed norms~\cite{Plastria:1992}.

One line of research in a multidimensional setting focuses on the study of robustness under asymmetric distances in phylogenetics.
The Fermat–Weber problem for phylogenetic trees provides a method for constructing consensus trees~\cite{Bryant:2003}.
The approach using asymmetric distances from~\cite{TropMedian} demonstrated improved properties compared to its symmetric counterpart~\cite{Lin+Yoshida:2018}.
However, its application to the Apicomplexan data set from Kuo~et~al.~\cite{apicomplexa} yielded unsatisfactory results.
The findings of this paper offer an explanation: the data set contains many outliers, and the tropical median from~\cite{TropMedian} has a low breakdown point, making it sensitive to such outliers.

Furthermore, as stated in \cite[Proposition~30]{Comaneci:2024}, the application discussed above connects robustness to majority-rule approaches in consensus problems~\cite{McMorris+Neumann}.
We will briefly revisit the relationship between breakdown points and majority rules in location problems in section~\ref{sec:maj-rule}.

Below we have a brief introduction to the main terms that will be used in the paper.

\subsection{Fermat--Weber problems} \label{subsec:FW}

We consider finite weighted sets $(A,w)$ of points in $\RR^d$, where $A\subset\RR^d$ is a non-empty finite set, and $w:\RR^d\to\RR^+$ a function such that $w_x>0$ if and only if $x\in A$.
In this case, $A$ serves as the support of $w$, i.e. $A=\{x\in\RR^d:w_x\neq 0\}$.
These sets can be viewed as collections of points with assigned multiplicities.

We can combine two weighted sets $(A,w)$ and $(B,v)$ by defining their sum as $(A,w)+(B,v):=(A\cup B,w+v)$, as described in the destination set algebra from \cite[\textsection 3.3]{Plastria2020}.
We say that $(A,w)$ is a subsample of $(B,v)$, denoted by $(A,w)\leq(B,v)$, if $A\subset B$ and $w_a\leq v_a$ for every $a\in A$.
In this situation we can subtract $(A,w)$ from $(B,v)$ to obtain $(B,v)-(A,w):=(C, v-w)$, where $C=\{x\in B:v_x>w_x\}$.

The Fermat--Weber problem for the weighted set $(A,w)$ is the minimization problem:
\begin{equation} \label{eq:FWproblem}
    \text{minimize}_{x\in\RR^d}\quad \sum_{a\in A}w_a d(a,x)
\end{equation}
where $d$ is a distance measure on $\RR^d$.
Any optimal solution to this problem will be called a Fermat--Weber point for the weighted set $(A,w)$.

In this paper, we focus on the case $d(x,y):=\gamma(y-x)$, with $\gamma$ being a gauge as defined in section~\ref{sec:gauge}.
We will also use the notation $\FW_{\gamma}(A,w)$ the set of all Fermat--Weber points for $(A,w)$, i.e. the minima of~\eqref{eq:FWproblem} when $d$~is the gauge distance induced by~$\gamma$.

\subsection{Breakdown points} \label{subsec:breakdown}
The \emph{breakdown point} of an estimator is a common way to measure how robust an estimator is to the presence of outliers.
In general, the breakdown point is a local property, but the estimators coming from \eqref{eq:FWproblem} for gauge distances are equivariant under translations of $A$, so we will see it as a global property.

In our setting, we will define the breakdown point of a location estimator $\xi$ which is applied to finite weighted samples of the
 form $(A,w)$.
We will denote by $w_A$ the sum of all weights of $A$, i.e. $w_A:=\sum_{a\in A}w_a$.
For any positive real $\tau<1$, we say that $(A',w')$ is a $\tau$-corruption of $(A,w)$ if it is obtained by replacing a subsample $(C,{\bar{w}})\leq(A,w)$ with ${\bar{w}}_C\leq\tau\cdot w_A$ by a sample $(C',{\bar{w}}')$ with the same weight.
In other words, $(A',w')=\big((A,w)-(C,{\bar{w}})\big)+(C',{\bar{w}}')$ where ${\bar{w}}'_{C'}={\bar{w}}_C$ and ${\bar{w}}_C/w_A\leq\tau$.

\begin{definition} \label{def:breakdown}
    We say that an estimator $\xi$ is \emph{$\tau$-robust} if for every weighted set $(A,w)$ there exists a bounded set $K_{\tau,A,w}\subseteq\RR^d$ such that for every $\tau$-corruption $(A',w')$ of $(A,w)$ we have $\xi(A',w')\in K_{\tau,A,w}$.

    We define the \emph{breakdown point} of the estimator $\xi$ as the supremum of all non-negative real numbers $\tau<1$ such that $\xi$ is $\tau$-robust. An estimator will be called \emph{robust} if the breakdown point is strictly positive.
\end{definition}

In other words, for any corruption below the breakdown point, the estimate of the corrupted sample is at a bounded distance (depending on the size of the corruption) from the original value.
However, this bound usually grows to infinity as the size of the corruption gets closer to the breakdown point as we will see in section~\ref{sec:unifRobust}.

\begin{remark} \label{rmk:breakdown-nonunique}
    Definition~\ref{def:breakdown} is stated for an estimator $\xi$ which is typically assumed to yield a unique result.
    However, in cases where the estimator is given by an optimization problem, such as \eqref{eq:M-estim} or \eqref{eq:FWproblem}, we can consider $\xi(A',w')$ as the set of all solutions and the condition $\xi(A',w')\in K_{\tau,A,w}$ should be interpreted as the entire solution set being contained within the bounded set $K_{\tau,A,w}$, i.e. $\xi(A',w')\subseteq K_{\tau,A,w}$. This ensures that even if multiple estimates are possible under corruption, all of them remain within a bounded region.
\end{remark}

However, the notion of breakdown point alone does not control how estimates behave as the level of corruption varies below that threshold. To capture a stronger form of robustness, we introduce the concept of \emph{uniform robustness}.

\begin{definition} \label{def:uniformRobust}
    We say that an estimator $\xi$ with breakdown point $\tau^\star>0$ is \emph{uniformly robust} if for every weighted sample $(A,w)$ there is a bounded set $K_{A,w}$ such that for every $\tau<\tau^\star$ and every $\tau$-corruption $(A',w')$ of $(A,w)$ we have $\xi(A',w')\in K_{A,w}$.
\end{definition}

\begin{remark} \label{rmk:uniformRobust-nonunique}
    Similar to the concept of $\tau$-robustness, if the estimator $\xi$ yields a set of values (e.g., the set of Fermat--Weber points), then the uniform robustness property implies that for every $\tau<\tau^\star$ and every $\tau$-corruption $(A',w')$, the entire set $\xi(A',w')$ remains within the bounded set $K_{A,w}$. This boundedness of the solution set is independent of the specific level of corruption $\tau$ (as long as it is below the breakdown point~$\tau^\star$).
\end{remark}

Note that the main difference between definitions~\ref{def:breakdown} and~\ref{def:uniformRobust} is that the bounded set from the latter does not depend on~$\tau$.
For uniformly robust estimators, we have a finite guarantee for the deviation of the corrupted estimate from the original value as long as the corruption remains below the breakdown point.

We want to specify that some sources, such as \cite[Chapter~11]{Huber1981}, define breakdown points in terms of \emph{contaminations} of samples and not \emph{corruptions}, as we presented above.
These two points of view are equivalent, but each have its advantages in theoretical results.
For this reason, we will focus on contaminations starting with section~\ref{sec:unifRobustNorm}.
Formally, for any positive real $\tau<1$, a $\tau$-contamination of a sample $(A,w)$ is a weighted set $(A,w)+(C,{\bar{w}})$ where ${\bar{w}}_C\leq\tau\cdot w_A/(1-\tau)$.
Note that the last condition is equivalent to ${\bar{w}}_C/({\bar{w}}_C+w_A)\leq\tau$, so $(A,w)+(C,{\bar{w}})$ could be seen as a $\tau$-corruption of a set ${(A,w)+(b,{\bar{w}}_C)}$ where $b$ is a point from the original, uncorrupted sample.
In the previous sentence, we used $(b,{\bar{w}}_C)$ as a shorthand notation for ${(\{b\},{\bar{w}}_C\cdot\delta_{\{b\}})}$, to increase the readability of the text, and we will continue using it when the situation occurs.
Conversely, any $\tau$-corruption is a $\tau$-contamination of the uncorrupted part.

\subsection{Outline} \label{subsec:outline}
In this paper we establish more precisely the breakdown point of a Fermat--Weber point as equal to $1/(1+\sigma)$ in case the distance is measured by a gauge with asymmetry measure $\sigma$, as defined in the next section.
This is done in two steps: first, in section~\ref{sec:upperBound}, we show that this estimator is not $\tau$-robust for any $\tau>1/(1+\sigma)$ and, in section~\ref{sec:lowerBound}, that it is $\tau$-robust below this limit.
This will include an explicit construction of a bounded set $K_{\tau,A,w}$ as defined above.
These results are formalized in Theorems~\ref{th:upperBound} and~\ref{th:lowerBound}, with some upper bounds also following from Remark~\ref{rmk:quantitativeRobustness}.
The breakdown point might not be enough to understand the effect of outliers in practical settings, as quoted in \cite{Diakonikolas+Kane:2023}:
\begin{quote}
    if one cares about the size of the errors that one incurs (more precisely than simply knowing whether or not they are finite), the breakdown point will be an insufficient measure of robustness.
\end{quote}
Therefore, specific constructions of $K_{\tau,A,w}$ give us more certainty about our estimator, which are deduced in the proofs of the aforementioned results.

In \cite{Plastria2016} the question was left open whether the same bounded set $K$ may contain all Fermat--Weber points for samples corrupted by any $\tau$ below the breakdown point, a property we named uniform robustness. 
We show in section~\ref{sec:unifRobust} that this does not hold in many cases, including the Euclidean distance, but that any polyhedral distance has uniform robustness for its Fermat--Weber points.
Thus, we guarantee a consistent performance in robustness for polyhedral Fermat--Weber points across varying contamination levels, making them more predictable in worst-case or uncertain conditions.
What is more, the displacement of the estimator of the corrupted sample is easier to analyze, especially when relating subsequently to majority rules in location problems.

In section~\ref{sec:unifRobustNorm} we concentrate on the uniform robustness for norms, i.e. the symmetric case.
In particular, we show that this is related to the boundedness of the elementary hull from Definition~\ref{def:EH}.
Using this characterization, we show that locally strictly convex norms are not uniformly robust.
Section~\ref{sec:CL} focuses on the asymmetric case, showing that the situation is then more complex.

In the last two sections, a different view on $\tau$-robustness provides an easier treatment: the \emph{contamination} approach to breakdown points of finite samples; see \cite[Chapter~11]{Huber1981}.
In particular, we introduce the set of all possible contaminations under a given threshold and present a few topological properties.
The latter are useful for computing contamination loci, as we will see in Example~\ref{eg:quantileCL}.

Section~\ref{sec:maj-rule} will be more informal, but discusses how some majority rules from location problems relate to robustness of location to outlier observations.
In particular, we discuss how this observation, together with the study of elementary hulls, was used in \cite{Comaneci:2024} to relate specific location problems to majority-rule consensus methods as defined in \cite{Margush+McMorris,McMorris+Neumann}.

The paper terminates with some final remarks.


\section{Gauges and skewness} \label{sec:gauge}
We always consider a finite gauge $\gamma$ on $\RR^d$ with finite dual gauge, i.e. $\gamma$ defines (and is defined by) a unit ball $B_{\gamma}:=\SetOf{x\in\RR^d}{\gamma(x)\leq 1}$ that is
convex, compact and has the origin $\0$ in its interior.   
For the basic properties and subdifferentials of a gauge $\gamma$ and its dual gauge $\gamma^\circ$, see, e.g. \cite{FundConvAnal,Plastria:2009}.

In particular, we will need the following.
For all $x,y\in\RR^d$ and $\lambda> 0$ we have
\begin{itemize}
    \item $\gamma(x)\geq 0$;
    \item $\gamma(x)=0 \Leftrightarrow x=\0$;
    \item $\gamma(\lambda x)=\lambda\gamma(x)$;
    \item $\gamma(x+y)\leq\gamma(x)+\gamma(y)$.
\end{itemize}
The dual gauge is defined by $\gamma^\circ(p):=\max_{x\in\RR^d}\SetOf{\langle p,x \rangle}{\gamma(x)=1}$, implying that for any $p,x\in\RR^d$
\[\langle p,x \rangle \leq \gamma^\circ(p)\gamma(x)\text{ (generalized Cauchy--Schwarz)}\] and $\gamma^{\circ\circ}=\gamma$.
The boundary of $\gamma^\circ$'s unit ball is $\partial B_{\gamma^\circ}=\SetOf{p\in\RR^d}{\gamma^\circ(p)=1}$.
The subdifferential of $\gamma$ at a point $x\neq\0$ is $\partial \gamma(x)=\SetOf{p\in\partial B_{\gamma^\circ}}{\langle p,x \rangle=\gamma(x)}$ and satisfies $\partial\gamma(\lambda x)=\partial\gamma(x)$
for all $\lambda>0$. Note that $\gamma^\circ(p)=1$ for all $p\in\partial\gamma(x)$ as soon as $x\neq\0$, while for $x,p\in\RR^d$ with 
$\gamma(x)=\gamma^\circ(p)=1$ we have $p\in\partial \gamma(x)$ if and only if $x\in\partial \gamma^\circ(p)$.
Hence, for non-zero vectors $p,x$ we have 
\begin{equation}\label{eq:dualsubdiff}
    \frac{p}{\gamma^\circ(p)}\in\partial\gamma(x) 
    \iff \frac{x}{\gamma(x)}\in\partial\gamma^\circ(p)
\end{equation}
The subdifferential of a gauge at the origin equals the unit ball of the dual norm, i.e. $\partial\gamma(\0)=B_{\gamma^\circ}$ and $\partial\gamma^\circ(\0)=B_{\gamma}$.

We will denote by $\sigma$ the \emph{skewness} of the gauge $\gamma$ as defined in \cite{Plastria:1992,Plastria:2009}, i.e.
\[\sigma:=\max_{x\neq\mathbf 0}\frac{\gamma(x)}{\gamma(-x)}=\max_{-x\in\partial B_{\gamma}}{\gamma(x)}=\max_{x\in- B_{\gamma}}{\gamma(x)}.\]

We call a \emph{skewness direction} of $\gamma$ 
any $\gamma$-unit direction $v$ on which the gauge $\gamma$ reaches it skewness, i.e. $v\in\RR^d$ such that $\gamma(v)=\sigma\cdot\gamma(-v)=1$.
The set $\SD_\gamma$ of skewness directions of $\gamma$ is nonempty by the continuity of $\gamma$ on the 
 compact boundary  $\partial B_{\gamma}$. Clearly $\sigma\geq 1$ and the lower bound 1 is reached if and only if $\gamma$ is symmetric, i.e. a norm,
 in which case $\SD_\gamma=\partial B_{\gamma}$. 
 
 Note that $-\sigma v\in \partial B_{\gamma}\subset B_{\gamma}$ for any $v\in \SD_\gamma$.
 Since the maximum of the convex function $\gamma$ on the closed convex set $- B_{\gamma}$ is reached at some extreme point $u\in \partial(-B_{\gamma})=-\partial B_{\gamma}$, there always exists a $u\in \SD_\gamma$ with $-\sigma u$ an extreme point of $B_{\gamma}$.
 
Figure~\ref{fig:ballsSD} displays the unit balls of three gauges in $\RR^2$, all of skewness 2, and their respective skewness directions $v$ as bold arrows as well as their corresponding $-\sigma v$ by dotted line segments; the following cases are shown:
\begin{description}
    \item[(a)] is a decentered circle with singleton $\SD_\gamma$;
    \item[(b)] is a centered equilateral triangle having three skewness directions; and
    \item[(c)] is a quadrilateral that shows that $\SD_\gamma$ may be non-discrete (here it contains a full interval of $\partial B_{\gamma}$ shown in bold).
\end{description}

\begin{figure}
    \centering
    \begin{tikzpicture}[scale=0.1];
    \begin{scope}[shift={(-40,0)}]
    \coordinate (o) at (0,0); \draw[fill] (o)  circle(0.2) node[below] {$\0$};
   \draw (5,0) circle[radius=15];
    \draw[very thick,->] (o) -- (-10,0); 
    \draw[dotted] (o)--(20,0);   
    
    \draw (5,-18) node {(a)};
  \end{scope}
   \begin{scope}
    \coordinate (top) at (0,20);     
    \path (-30:20) coordinate (right);
    \path (210:20) coordinate (left);
    \draw (top) -- (right) -- (left) -- cycle;
     \coordinate (o) at (0,0); \draw[fill] (o)  circle(0.2) node[below left] {$\0$};
   
    \draw[very thick,->] (o) -- ($(top)!0.5!(right)$);
    \draw[very thick,->] (o) -- ($(left)!0.5!(right)$);
    \draw[very thick,->] (o) -- ($(top)!0.5!(left)$);
    \draw[dotted] (o)--(top) (o)--(right) (o)--(left) ;   
    
    \draw (0,-18) node {(b)};
  \end{scope}
   \begin{scope}[shift={(35,0)}]
    \coordinate (o) at (0,0); \draw[fill] (o)  circle(0.2) node[above] {$\0$};
    \path (-10,20) coordinate (rt);
    \path (-10,-10) coordinate (rb);
    \path (20,15) coordinate (lt);
    \path (20,-10) coordinate (lb);
    \draw (rt) -- (rb) -- (lb) -- (lt) -- cycle;

\path[name path=dlt] (lt) -- ($(lt)!2!(o)$);
\path[ name path=dlb] (lb) -- ($(lb)!2!(o)$);
\path[name path=right] (rt) -- (rb);
\path[name intersections={of=right and dlt} ] (intersection-1) coordinate (olt);
\path[name intersections={of=right and dlb} ] (intersection-1) coordinate (olb);

\path[ name path=drt] (rt) -- ($(rt)!2!(o)$);
\path[name path=bot] (lb) -- (rb);
\path[name intersections={of=bot and drt} ] (intersection-1) coordinate (ort);
    \draw[very thick,->] (o) -- (ort); 
    \draw[dotted] (rt)--(o);   

    \draw[very thick,->] (o) -- (olt); 
    \draw[very thick,->]  (o) -- (olb) ;
    \draw[very thick,->]  (o) -- ($(rt)!(o)!(rb)$) ;
    \draw[ultra thick, opacity=0.7]  (olt) -- (olb) ;
    \draw[dotted] (lt)--(o)--(lb);   
    
    \draw (5,-18) node {(c)};
  \end{scope}
  \end{tikzpicture}    
  \caption{Three unit balls of gauges of skewness 2 and their skewness directions (bold arrows or bold segments; see text for details)}
    \label{fig:ballsSD}
\end{figure}

\begin{example}
An important class of gauges comprises skewed norms, introduced in~\cite{Plastria:1992}, which encompass various asymmetric distances encountered when flying in wind, navigating in current, or transporting on inclined terrain and also when intercepting a moving target \cite{Cera2007}.
These are defined by $\gamma(x)=\nu(x)-\langle x,u\rangle$ where $\nu$ is some norm on $\RR^d$ and $u$ is some vector in $\RR^d$ satisfying $\nu^\circ(u)<1$. The skewness of this gauge was obtained in~\cite{Plastria:1992} as 
\[\sigma=\frac{1+\nu^\circ(u)}{1-\nu^\circ(u)}.\]
Skewed norms are the duals of gauges with centrally symmetric unit balls around the point $-u$.
It is not hard to see that for such a skewed norm $\gamma$ 
one has
 \[\SD_\gamma=\SetOf{\frac{-q}{1+\nu^\circ(u)}}{q\in \partial\nu^\circ(u)}.\]  
\end{example}

\begin{lemma} \label{lem:subdiffskew}  For any gauge $\gamma$ with skewness $\sigma$, its dual gauge $\gamma^\circ$ also has skewness $\sigma$ and we have for all $v\in\RR^d$:
\begin{equation}\label{eq:inSDgamma}
\mbox{ $v\in \SD_\gamma$ if and only if  } -\frac{1}{\sigma}\partial\gamma(v)\subset\partial\gamma(-v).
\end{equation}

Moreover, 
\begin{equation}\label{eq:SDgammadual}
\SD_{\gamma^\circ} = -\frac{ 1}{\sigma}\bigcup_{v\in\SD_\gamma}\partial\gamma(v).
\end{equation}
\end{lemma}

\begin{proof}
Equality of skewness between  $\gamma$ and  $\gamma^\circ$ was shown in \cite[Corollary~11]{Plastria:2009}.

To show (\ref{eq:inSDgamma}) consider first
any $v\in \SD_\gamma$, and let $r\in\partial\gamma(v)$, which means that
\begin{description}
    \item[(a)] $\langle r,v\rangle=\gamma(v)$, and
    \item[(b)] $\gamma^\circ(r)=1$. 
\end{description}
We also recall that $\gamma(v)=1$ since $v$ is a skewness direction. Thus, $\langle r,v\rangle=1$. 

We have to show that $-\frac{ r}{\sigma}\in \partial \gamma(-v)$. To this end it suffices to show that
\begin{description}
    \item[(A)] $\langle-\frac{ r}{\sigma},-v\rangle=\gamma(-v)$, and
    \item[(B)] $\gamma^\circ(-\frac{ r}{\sigma})= 1$.
\end{description}

It follows from (a) that $\langle-\frac{ r}{\sigma},-v\rangle=\frac{ 1}{\sigma}\langle r,v\rangle=\frac{ 1}{\sigma}=\gamma(-v)$, hence we have (A).

By the generalized Cauchy--Schwarz inequality we also have \[\gamma(-v)=\left\langle-\frac{ r}{\sigma},-v\right\rangle\leq \gamma^\circ\left(-\frac{ r}{\sigma}\right)\gamma(-v),\] so $\gamma^\circ(-\frac{ r}{\sigma})\geq 1$.
From (b) and the fact that $\gamma^\circ$ has skewness $\sigma$, 
we find that \[\gamma^\circ\left(-\frac{ r}{\sigma}\right)\leq \sigma \gamma^\circ\left(\frac{ r}{\sigma}\right)=\gamma^\circ(r)=1,\] yielding (B). 

Conversely, for the 'if' part of (\ref{eq:inSDgamma}), consider any $r\in\partial\gamma(v)$ (which must exist) and thus, by hypothesis,  $-\frac{r}{\sigma}\in\partial\gamma(-v)$.  Then $ 1=\gamma(v)=\langle r,v\rangle=\sigma \langle -\frac{ r}{\sigma},-v\rangle=\sigma\gamma(-v)$ and hence $v\in\SD_\gamma$.

For (\ref{eq:SDgammadual}) we start by   showing the inclusion 
\begin{equation}\label{eq:SDgammadualcontains}
\SD_{\gamma^\circ} \supset -\frac{ 1}{\sigma}\bigcup_{v\in\SD_\gamma}\partial\gamma(v).
\end{equation}
Let $v\in \SD_\gamma$ and  $r\in\partial\gamma(v)$. Above we  showed in (B) that $\gamma^\circ(-\frac{ r}{\sigma})=1$.
Next, we have $\sigma\gamma^\circ(\frac{ r}{\sigma})=\gamma^\circ(r)=1$.
Since  $\gamma^\circ$ has skewness $\sigma$ this means that 
$-\frac{ r}{\sigma}\in\SD_{\gamma^\circ}$, yielding the sought inclusion.

To obtain the inverse inclusion in  (\ref{eq:SDgammadual}) we have to show that any $s\in\SD_{\gamma^\circ}$
is of the form $-\frac{ r}{\sigma}$ with $r\in\partial\gamma(v)$ for some 
$v\in \SD_\gamma$. To this end note first that $\gamma^\circ(-s)=\frac{1}{\sigma}$, so calling $r:=-\sigma s$ we have $\gamma^\circ(r)=1$. Secondly take any $w\in\partial\gamma^\circ(s)$ and define 
$v:=-\frac{ w}{\sigma}$. 
Note now that   (\ref{eq:SDgammadualcontains}) may be applied to the gauge 
$\gamma^\circ$, which, using $s\in\SD_{\gamma^\circ}$ and $w\in\partial\gamma^\circ(s)$, implies  we have $v=-\frac{ w}{\sigma}\in\SD_{\gamma^{\circ\circ}}=\SD_{\gamma}$.
Applying also   (\ref{eq:inSDgamma}) to the gauge 
$\gamma^\circ$, we see that $s\in\SD_{\gamma^\circ}$ and $w\in\partial\gamma^\circ(s)$ implies , 
$v=-\frac{ w}{\sigma}\in\partial\gamma^\circ(-s)=\partial\gamma^\circ(r)$, 
or, equivalently (since $\gamma^\circ(r)=1$), $r\in\partial\gamma(v)$.
Noting, finally, that $-\frac{ r}{\sigma}=s$, the proof is complete.
\end{proof}

\begin{remark}
The inclusion in (\ref{eq:inSDgamma}) of Lemma~\ref{lem:subdiffskew} might be strict.
Consider, for example, a gauge $\gamma$ with as unit ball $B_{\gamma}$ a simplex in $\RR^d$  ($d\geq 2$) having $\0$ in its interior.
As previously observed, the skewness is realized at some \( v \) where \( -\sigma v \) is an extreme point in \( B_\gamma \). Thus, \( -\sigma v \) is a vertex of \( B_\gamma \), and the subdifferential of \( \gamma \) at \( -\sigma v \) (and thus at~\( -v \)) is \((d - 1)\)-dimensional. In contrast, \( v \) lies in the relative interior of the opposite simplex facet, where the subdifferential of \( \gamma \) is a singleton.
\end{remark}

\begin{tcolorbox}[width=\linewidth, sharp corners=all, colback=white!80!black, boxrule=0pt]
In all what follows we assume that $\gamma$ is a finite gauge with skewness $\sigma$ and  $v\in \SD_\gamma$. 
\end{tcolorbox}

\section{The upper bound on the breakdown point} \label{sec:upperBound}

Consider the Fermat--Weber problem with distance derived from $\gamma$: find $x$ minimizing
\begin{equation} \label{FWproblem}
    x\mapsto\sum_{a\in A}w_{a}\gamma(x-a)
\end{equation}
where all $w_{a}>0$.
In other words, we study the Fermat--Weber problem on a weighted set $(A,w)$ where $w:A\to\RR_{>0}$ is the function assigning weights to the points of $A$.

The following corollary of Lemma~\ref{lem:subdiffskew} shows that the  Fermat--Weber problem on a line with highest skewness reduces to the one-dimensional case (quantiles).

\begin{lemma} \label{cor:FWskewline}
    If all sample points $A$ of a Fermat--Weber problem lie on a line whose direction is parallel to some $v\in\SD_\gamma$, then there is a Fermat--Weber point on that line.
\end{lemma}

\begin{proof}
     For simplicity, we will assume that all weights are equal; a similar proof can be done for the general case.
    After a translation, we can also assume that $A=\{\tau_1 v,\dots,\tau_n v\}$ with ${\tau_1<\dots<\tau_n}$.
    Setting $\ell=\lceil n/(1+\sigma)\rceil$, we show that $\tau_\ell v$ is a Fermat--Weber point.
    
 For the $\ell-1$ values $i<\ell$ the term $\tau_\ell-\tau_i$ is positive, so $\partial\gamma\left((\tau_\ell-\tau_i)v\right)=\partial \gamma(v)$, while for the $n-\ell$ values $i>\ell$ it is negative, so $\partial\gamma\left((\tau_\ell-\tau_i)v\right)=\partial \gamma(-v)$.
     Therefore, the subdifferential of $\sum_{a\in A}\gamma(\cdot-a)$ at $\tau_\ell v$ equals
    \begin{equation} \label{eq:subdiff_incl}
        \sum_{i=1}^n\partial\gamma\left((\tau_\ell-\tau_i)v\right)= (\ell-1)\partial \gamma(v)+\partial\gamma(\mathbf 0)+{(n-\ell)}\partial\gamma(-v).
    \end{equation}

    To clarify, this equality holds because for a convex set $C$, we have $kC=\underbrace{C + \cdots + C}_{k\text{ times}}$, so multiplying the subdifferential by a scalar is equivalent to summing it $k$ times; see \cite[Theorem~3.2]{Rockafellar:70}.

    Consider then any $r\in\partial\gamma( v)$. We know from Lemma~\ref{lem:subdiffskew} that $-\frac{1}{\sigma} r\in\partial\gamma(-v)$, so $r$ and $-\frac{1}{\sigma} r$ both belong to the (convex) unit dual ball $B_{\gamma^\circ}$.
    The value $\alpha=\frac{n-\ell}{\sigma}-\ell+1$ satisfies $-1/\sigma<\alpha\leq 1$, and therefore  $\alpha r\in B_{\gamma^\circ}=\partial\gamma(\mathbf 0)$.

    But ${(\ell-1) r}+\alpha r+{(n-\ell)}\left(-\frac{1}{\sigma} r\right)=\mathbf 0$, and 
 equation (\ref{eq:subdiff_incl})   shows  that this zero-vector lies in the subdifferential of $\sum_{a\in A}\gamma(\cdot-a)$ at $\tau_\ell v$.
Since the objective function of the Fermat--Weber problem is convex, this means that $\tau_\ell v$ minimizes it, so it is a Fermat--Weber point.
\end{proof}

\begin{remark}
Lemma~\ref{cor:FWskewline} may also be proven using a lengthier but quite elementary ideal point proof (without any convex analysis) similar to section~3.7 in \cite{Plastria2020}.
\end{remark}

\begin{remark}
    In the asymmetric case, the solutions of the Fermat--Weber problem with $A$ contained in a line might lie outside the line.
    For example, the case of a simplicial gauge is analyzed in \cite{TropMedian} and it is shown that the Fermat--Weber points lie in the tropical convex hull of the input points.
    Actually, the tropical convex hull coincides with the elementary hull for a simplicial gauge as will be defined in section~\ref{sec:unifRobust}, Definition~\ref{def:EH}.
    Unless the points lie along a skewness direction, most Fermat--Weber points will generally lie outside the line; cf.~\cite[Example~8]{TropMedian}.
The above result shows that the skewness directions are special.
In the symmetric case, all unit directions are skewness directions, so the result holds for every line.
\end{remark}

To show that the upper bound on the breakdown point is $1/(1+\sigma)$, we will move a part of the sample, whose total weight is larger than the aforementioned value, in the direction $-v$.
We will then show that the Fermat--Weber point of the new weighted set can be arbitrarily far from the original location.

We will split the set $A$ in two: the set $C$ of corrupted points and the set $D$ of uncorrupted points. Considering also the weights, we have $(A,w)=(D,{\bar{w}}')+(C,{\bar{w}})$.
The points of $C$ will be moved in the direction of $-v$, so they are replaced by $C-Mv$ for some large $M>0$, while the weights are preserved.
Hence, the objective function
\[f_M(x)=\sum_{d\in D}{\bar{w}}'_d\gamma(x-d)+\sum_{c\in C}{\bar{w}}_c\gamma(x-c+Mv)\]
depends on $M$.
We can split it as $f_M(x)=g(x)+h_M(x)$ where
\[g(x):=\sum_{d\in D}{\bar{w}}'_d\gamma(x-d) \quad\text{and}\quad h_M(x):=\sum_{c\in C}{\bar{w}}_c\gamma(x-c+Mv).\]
As introduced before, for a subset $S$ of $A$, we denote by $w_S$ the sum $\sum_{s\in S}w_s$.

We will prove that $\mathbf 0\notin\partial g(u)+\partial h_M(u)$ for every $u$ in  some bounded set $U$ and $M$ sufficiently large, 
as soon as the weight ratio ${\bar{w}}_C/w_A$ is larger than $1/(1+\sigma)$.
For this we need two lemmas bounding the  dual norm $\gamma^\circ$ of subgradients in $\partial g$ and $-\partial h_M$.

\begin{lemma} \label{lem:part1}
    For every $x\in\RR^d$ and $p\in\partial g(x)$, we have $\gamma^\circ(p)\leq {\bar{w}}'_D$.
\end{lemma}

\begin{proof}
    The point $p$ is of the form $p=\sum_{d\in D}{\bar{w}}'_d p_d$ with $p_d\in\partial\gamma(x-d)$.
    Then $\gamma^\circ(p_d)\leq 
    1$ for every $d\in D$ and
    \[\gamma^\circ(p)=\gamma^\circ\left(\sum_{d\in D}{\bar{w}}'_d p_d\right)\leq\sum_{d\in D}{\bar{w}}'_d\gamma^\circ(p_d)\leq 
    \sum_{d\in D}{\bar{w}}'_d\]
    from the sublinearity of $\gamma^\circ$.
\end{proof}

\begin{lemma} \label{lem:part2}
    Let $U\subset\RR^d$ be a bounded set and $\tau>0$ a positive constant.
    Then there exists $M_{U,\tau}>0$ such that $\gamma^\circ(-q)>\sigma {\bar{w}}_C-\tau$ for every $M>M_{U,\tau}$ and $q\in\partial h_M(U):=\bigcup_{u\in U}\partial h_M(u)$.
\end{lemma}

\begin{proof}
    Since $\frac{x+Mv}{\gamma(x+Mv)}\to v$ as $M\to\infty$ for every $x\in\RR^d$ and $\partial\gamma$ is invariant under positive scalings, we will approximate $\partial\gamma(u-c+Mv)$ by $\partial\gamma(v)$ for every $u\in U$ and ${c\in C}$.
    This will be obtained using the outer semi-continuity of $\partial\gamma$ \cite[Theorem~D.6.2.4]{FundConvAnal}, in particular at $v$, which means that for every $\varepsilon>0$, there exists $\delta_{\varepsilon}>0$ such that $\gamma(y-v)<\delta_{\varepsilon}$ implies $\partial\gamma(y)\subset\partial\gamma(v)+\varepsilon B_{\gamma^\circ}$.

    The continuity of $\gamma$ implies that the set
    \[S_{\varepsilon}=\left\{x\in\RR^d\setminus\{-v\}:\gamma\left(\frac{x+v}{\gamma(x+v)}-v\right)<\delta_{\varepsilon}\right\}\]
    is open.
    Remark that we have $\mathbf 0\in S_{\varepsilon}$ because we assumed $v$ to be a skewness direction, which entails $\gamma(v)=1$.
    Since $S_\varepsilon$ is open, there exists a ball $\eta B_{\gamma}$ around $\mathbf 0$ contained in $S_{\varepsilon}$, where $\eta>0$.
    But $U$ is bounded and $C$ is finite, so $U-C=\{u-c:u\in U,c\in C\}$ is still bounded, and there exists $M_{\varepsilon}>0$ such that for all $M>M_{\varepsilon}$\[\frac{1}{M}(U-C)\subset \eta B_{\gamma}\subset S_\varepsilon.\]
    We should mention that we consider $M_\varepsilon$ large enough such that $u-c+Mv\neq\0$ for all $M>M_\varepsilon$, which is possible since $v$ is a non-zero vector.

    Consider now  any $u\in U$ and $M>M_{\varepsilon}$.    
    For any $c\in C$ we have $(u-c)/M\in S_{\varepsilon}$, which means
    that 
    \[\gamma\left(\frac{u-c+Mv}{\gamma(u-c+Mv)}-v\right)<\delta_{\varepsilon}\]
    so also
    $\partial\gamma(u-c+Mv)=\partial\gamma\left(\frac{u-c+Mv}{\gamma(u-c+Mv)}\right)\subset\partial\gamma(v)+\varepsilon B_{\gamma^\circ}$. 
    Thus after weighted summation we find 
    \begin{equation} \label{incl:approxSubDiff}
        \partial h_M(u)\subset {\bar{w}}_C\partial\gamma(v)+\varepsilon {\bar{w}}_C B_{\gamma^\circ}.
    \end{equation}

    Now for any $p\in\partial\gamma(v)$,
    by Lemma~\ref{lem:subdiffskew} we have  $-\frac{p}{\sigma}\in \partial\gamma(-v)$.
    Consequently, $\gamma^\circ(-\frac{p}{\sigma})=1$.
    Therefore, $\gamma^\circ(-p)=\sigma$ and
    \(\gamma^\circ\left(-{\bar{w}}_C\cdot p\right)={\bar{w}}_C\cdot {\gamma^\circ(-p)}=\sigma {\bar{w}}_C.\)
    
Thus $-{\bar{w}}_C\partial\gamma(v)$ is a compact subset of $W:=(\gamma^\circ)^{-1}\left(\left(\sigma {\bar{w}}_C-\tau,+\infty\right)\right)$
which is open by the continuity of $\gamma^\circ$.
So $\varepsilon>0$ may be chosen such that \[-{\bar{w}}_C\partial\gamma(v)-\varepsilon {\bar{w}}_C B_{\gamma^\circ}\subset W\] 
and inclusion \eqref{incl:approxSubDiff} shows that $-\partial h_M(u)\subset W$ for any $u\in U$ and $M>M_{\varepsilon}$.

Hence, for every $M>M_{\varepsilon}$ and $q\in\partial h_M(U)$ we will have  $\gamma^\circ(-q)>\sigma {\bar{w}}_C-\tau$,
which proves the lemma with $M_{U,\tau}:=M_{\varepsilon}$.
\end{proof}

\begin{theorem}\label{th:upperBound}
    Consider $(C,{\bar{w}})\leq (A,w)$ such that ${\bar{w}}_C/w_A>1/(1+\sigma)$ and $U$ some bounded subset of $\RR^d$.
    Then any solution of the  Fermat--Weber problem lies outside $U$ if we replace $c$ with $c-Mv$ for all $c\in C$ with $M$ sufficiently large. In other words, the breakdown point is at most $1/(1+\sigma)$.
\end{theorem}

\begin{proof}
    Set $\tau:=\sigma\cdot\left({\bar{w}}_C-{w_A}/{(1+\sigma)}\right)$ and $M_0:=M_{U,\tau}$ from Lemma~\ref{lem:part2}.
    We show that no point of $U$ can be a minimum of $f_M$ for any 
    $M>M_0$ by reductio ad absurdum.
    So we assume that there exists $u\in U$ and $M>M_0$ such that $f_M$ attains its minimum at $u$.
    Due to the convexity of $f_M$, this is equivalent to $\mathbf 0\in\partial f_M(u)=\partial g(u)+\partial h_M(u)$.
    Therefore, there exists $p\in\partial g(u)$ such that $-p\in\partial h_M(u)$.

    From Lemma~\ref{lem:part1}, we have
    \begin{equation} \label{ineq:contr1}
        \gamma^\circ(p)\leq {\bar{w}}'_D=w_A-{\bar{w}}_C<w_A-w_A/(1+\sigma)=w_A\cdot\sigma/(1+\sigma).
    \end{equation}
    On the other hand, Lemma~\ref{lem:part2} implies that
    \begin{equation} \label{ineq:contr2}
        \gamma^\circ(p)>\sigma {\bar{w}}_C-\tau=w_A\cdot\sigma/(1+\sigma).
    \end{equation}
    The relations \eqref{ineq:contr1} and \eqref{ineq:contr2} are contradictory, so our assumption was false.
    Accordingly, for any  
    $M>M_0$ no minimum of $f_M$ can be a point of $U$.
\end{proof}


\section{The lower bound on the breakdown point} \label{sec:lowerBound}

We now show that the Fermat--Weber point with gauge-distance  is $\tau$-robust for any $\tau<1/(1+\sigma)$.
We  adapt the proof of \cite[Theorem~2.2]{bdpoint} to the asymmetric case. 
This yields a much simpler proof than the adaptation of \cite[Theorem 6]{Plastria2016} to the asymmetric case
suggested in \cite{Plastria2023}, as no techniques from convex analysis are used.

We will use the notation:
\[B_\gamma(a;r)=\{x\in\RR^d:\gamma(x-a)\leq r\}=a+rB_{\gamma}.
\]

\begin{theorem} \label{th:lowerBound}
    Let $\gamma$ be a gauge with skewness $\sigma$, and let $(A,w)$ be a weighted sample. 
    Suppose $(A',w')$ is a $\tau$-corruption of $(A,w)$ for some $\tau < 1/(1+\sigma)$.
    Then there exists a bounded set $K_{\tau,A,w}$, depending only on $A$, $w$ and $\tau$, 
    such that all solutions of \eqref{FWproblem} on the corrupted sample $(A',w')$ are contained in $K_{\tau,A,w}$.
\end{theorem}

\begin{proof}
    Let $Y^\star$ be the set of minima of \eqref{FWproblem}. 
    Since $A$ is finite and $Y^\star$ is compact we may define  $M:=\max_{y^\star\in Y^\star}\max_{a\in A}\gamma({y^\star-a})$.  

    We will consider that the corruption of $(A,w)$ is given by $(A',w')=(A,w)-(C,{\bar{w}})+(C',{\bar{w}}')$ with ${\bar{w}}_C/w_A=\tau$, ${\bar{w}}'_{C'}={\bar{w}}_C$, and $C'\subset\RR^d$.
    
    We denote by $X^\star$ the set of minima of \eqref{FWproblem} with $(A,w)$ replaced by $(A',w')$.
    We will prove that there exists a positive real number $\kappa$, depending on $A$ and ${\bar{w}}_C$, but not on $A'$, such that \[\max_{y^\star\in Y^\star}\max_{x^\star\in X^\star}\gamma(y^\star-x^\star)\leq\kappa.\]
    After proving the existence of $\kappa$, our conclusion follows by setting
    \begin{equation} \label{eq:K_robustness}
        K_{\tau,A,w}:=\left\{x\in\RR^d:\max_{y^\star\in Y^\star}\gamma(y^\star-x)\leq\kappa\right\}
    \end{equation}
    which is bounded 
    because $Y^\star$ is a compact set.

    Select arbitrary points $y^\star\in Y^\star$ and $x^\star\in X^\star$.
    Moreover, let 
    \[\delta:=\inf\{\gamma(x^\star-u):u\in B_\gamma(y^\star;(1+\sigma)M)\}.\]
     $B_\gamma(y^\star;(1+\sigma)M)$ is compact and $\gamma$ is continuous, so $\delta=\gamma(x^\star-u')$ for some $u'\in B_\gamma(y^\star;(1+\sigma)M)$. 
    From the triangle inequality and the definition of $\sigma$, we obtain:
    \begin{equation} \label{eq:simpleBound}
    \begin{split}
        \gamma(y^\star-x^\star)&\leq \sigma\cdot\gamma(x^\star-y^\star)\\
        & \leq \sigma\cdot\left( \gamma(x^\star-u')+\gamma(u'-y^\star) \right) \\
        & = \sigma\cdot\delta + \sigma\cdot\gamma(u'-y^\star) \\
        & \leq \sigma\cdot\delta + \sigma(1+\sigma)M.
    \end{split}
    \end{equation}

    Let us first consider the case $\delta>0$, i.e. $x^\star\notin B_\gamma(y^\star;(1+\sigma)M)$.

    For any corrupted element $a\in C'$  we obtain
    \begin{equation}\label{eq:corr}
    \begin{split}
        \gamma(x^\star-a)&\geq\gamma(y^\star-a)-\gamma(y^\star-x^\star)\\
        &\geq \gamma(y^\star-a) - \left(\sigma\cdot\delta + \sigma(1+\sigma)M\right),
    \end{split}
    \end{equation}
    where we used the triangle inequality and \eqref{eq:simpleBound}.

\begin{figure}
\centering
   \begin{tikzpicture}[scale=0.3];
    \draw[fill] (0,0) coordinate (as) circle(0.2) node[above right] {$y^\star$}; 
    \draw[fill] (20,0) coordinate (xs) circle(0.2) node[below] {$x^\star$}; 
    \draw[fill] (2,-2) coordinate (a) circle(0.2) node[left] {$a$}; 
    \draw[name path=ball]  (-3,2).. controls (4,2)..(5,-1)--(4,-3)--(-2,-3) .. controls (-5,0) .. cycle; 
    \draw  (4,2)  node[right]{$B_\gamma(y^\star,(1+\sigma)M)$}; 
    \draw[name path=axs] (a)--(xs);
    \path[name intersections={of=ball and axs} ] (intersection-1) coordinate (u);
    \draw[fill] (u)  circle(0.2) node[below  right] {$u$}; 
   \draw[dotted] (u)--(as)--(a);

  \end{tikzpicture}
 \caption{Case  of an uncorrupted $a\in A$ in the proof of theorem \ref{th:lowerBound}}\label{fig:ProofTh}
\end{figure}

    For any uncorrupted element $a\in A$
    we have \(\gamma(y^\star - a) \leq M\) by the definition of~$M$, which implies 
    \(\gamma(a - y^\star) \leq \sigma \cdot M<(1+\sigma)M\).
    So $a$ is in the interior of the ball \( B_\gamma(y^\star; (1+\sigma)M) \) while \( x^\star \) is outside this ball. Therefore we may
    consider  the point of intersection $u$ between the boundary $\partial B_\gamma(y^\star;(1+\sigma)M)$ and the segment $[a,x^\star]$
    (see Figure \ref{fig:ProofTh}).
     We then have \(\gamma(u-y^\star)=(1+\sigma)M\) and
    \begin{equation}\label{eq:non-corr}
    \begin{split}
    \gamma(x^\star-a)& = \gamma(x^\star-u)+\gamma(u-a) \quad\text{($\gamma$ is linear over a line segment)}\\
    & \geq \delta + \gamma(u-a) \quad\text{(definition of $\delta$)}\\
    & \geq \delta + \gamma(u-y^\star)-\gamma(a-y^\star) \quad\text{(triangle inequality)}\\
    & = \delta + (1+\sigma)M - \gamma(a-y^\star) \\
    & \geq \delta + (1+\sigma)M - \sigma\cdot\gamma(y^\star-a)
    \quad\text{(skewness)}\\
    & \geq \delta + \gamma(y^\star-a) \quad\text{(definition of $M$)}.
    \end{split}
    \end{equation}

    By weighting inequalities \eqref{eq:corr} and \eqref{eq:non-corr} according to the corrupted and uncorrupted elements in $(A',w')$ respectively, and then adding the resulting inequalities, and using ${\bar{w}}'_{C'}={\bar{w}}_C$, we obtain
    \begin{equation} \label{ineq:main}
    \begin{split}
        \sum_{a\in A'}w'_{a}\gamma(x^\star-a) & \geq \sum_{a\in A'}w'_{a}\gamma(y^\star-a) + (w_A-{\bar{w}}_C)\delta-{\bar{w}}_C(\sigma\cdot\delta+\sigma(1+\sigma)M)\\
        & = \sum_{a\in A'}w'_{a}\gamma(y^\star-a) + [w_A-(1+\sigma){\bar{w}}_C]\delta-{\bar{w}}_C\sigma(1+\sigma) M.
    \end{split}
    \end{equation}

    From the definition of $x^\star$, we must have
    \[\sum_{a\in A'}w'_{a}\gamma(y^\star-a)\geq\sum_{a\in A'} w'_{a}\gamma(x^\star-a).\]
    Together with \eqref{ineq:main} and the hypothesis that $w_A>(1+\sigma){\bar{w}}_C$, we obtain:
    \[\delta\leq \frac{\sigma(1+\sigma){\bar{w}}_C}{w_A-(1+\sigma){\bar{w}}_C}M.\]
    We proved this bound under the assumption $\delta>0$, but it is  evidently also valid for $\delta=0$.
    Then \eqref{eq:simpleBound} gives
    \begin{equation} \label{eq:defKappa}
        \gamma(y^\star-x^\star)\leq \left(\frac{\sigma {\bar{w}}_C}{w_A-(1+\sigma){\bar{w}}_C}+1\right)\cdot \sigma (1+\sigma)M
        =:\kappa.
    \end{equation}
    This bound depends only on $A$ (from the definition of $M$) and ${\bar{w}}_C/w_A=\tau$, whereas $y^\star$ and $x^\star$ were chosen arbitrarily in $Y^\star$ and $X^\star$, respectively.
    
    Hence, $\max_{y^\star\in Y^\star}\max_{x^\star\in X^\star}\gamma(y^\star-x^\star)\leq\kappa$.
\end{proof}

\begin{remark} \label{rmk:quantitativeRobustness}
    The proof of Theorem~\ref{th:lowerBound} also provides an upper bound on the deviation between Fermat--Weber points under data corruption.
    More precisely, if $(A,w)$ is the original sample and $(C,{\bar{w}})$ is the corrupted part, then
    \begin{equation} \label{eq:quantitativeRobustness}
        \gamma(y^\star-x^\star)\leq \left(\frac{\sigma {\bar{w}}_C}{w_A-(1+\sigma){\bar{w}}_C}+1\right)\cdot \sigma (1+\sigma)\max_{{\scriptscriptstyle b\in \FW_\gamma(A,w)}}\max_{{\scriptscriptstyle a\in A}}\gamma(b-a)
    \end{equation}
    where $x^\star$ is any Fermat--Weber point for $(A,w)$ and $y^\star$ any Fermat--Weber point for its corruption.
    This comes from~\eqref{eq:defKappa}.

    The bound above is generally loose and may overestimate the true deviation.
    Tighter estimates are more difficult to obtain, but in some cases refinements are possible. For instance, when $\gamma$ is a polyhedral norm, one can exploit the structure of the elementary hull introduced in the next section (see Definition~\ref{def:EH} and Proposition~\ref{prop:elemSet}).
    Although deriving explicit constants from this approach is computationally involved, it provides a basis for improving the error bounds in structured settings.
\end{remark}

Putting Theorems \ref{th:lowerBound} and \ref{th:upperBound} together, we obtain the following result.
\begin{corollary} \label{cor:breakdownFW}
The Fermat--Weber point with gauge distance has breakdown point $1/(1+\sigma)$.

\end{corollary}

\section{Uniform robustness} \label{sec:unifRobust}

From Corollary~\ref{cor:breakdownFW} and Definition~\ref{def:uniformRobust}, a Fermat--Weber point is uniformly robust if for every weighted set $(A,w)$ there exists a bounded set $K_{A,w}$ such that the Fermat--Weber point of every $\tau$-corruption $(A',w')$ of $(A,w)$ belongs to $K_{A,w}$ whenever $\tau<1/(1+\sigma)$.

We already obtained in \eqref{eq:K_robustness} a set that contains
all Fermat--Weber points of any $\tau$-corruption of $(A,w)$.
More precisely, this set is $K=\FW_\gamma(A,w)+\kappa(-B_\gamma)$ where $\kappa$ comes from \eqref{eq:defKappa} taking $\tau={\bar{w}}_C/w_A$:
\[\kappa=\frac{1-\tau}{1-(1+\sigma)\tau}\cdot \sigma (1+\sigma)\max_{{\scriptscriptstyle b\in \FW_\gamma(A,w)}}\max_{{\scriptscriptstyle a\in A}}\gamma(b-a).\]
Note that altering the size of the corruption such that $\tau\nearrow 1/(1+\sigma)$ makes $\kappa$ go to infinity.
In particular, the closer $\tau$ is to the breakdown point $1/(1+\sigma)$, the larger the bounded set $K$ containing all the optima of the corrupted Fermat--Weber problems.
In other words, we can use the proof of Theorem~\ref{th:lowerBound} to check uniform robustness and we show, in the next example, that we cannot select $K$ independently of $\tau$ for every gauge.

\begin{example} \label{eg:3points}
    Let $A$ be a set of three points in $\RR^2$ and $\gamma$ the Euclidean norm $\|\cdot\|$ (of skewness $\sigma=1$).
    If there is a point, say $a\in A$, such that $w_a>w_A/2$, then $a$ will be the unique Fermat--Weber point whatever change we make to the other points, so the uniform bound $K=A$ can be taken in this case.
    However, this is impossible if we assume $w_a<w_A/2$ for all $a\in A$.

    To see this, fix $a\neq b$ from $A$ and choose $m$ arbitrarily in $\RR^2$, but outside the line through $a$ and $b$.
    We will prove that $m$ is a Fermat--Weber point of three points $a, b, c'$ where $c'$ is to be constructed.
    In other words, by appropriately choosing $c$ (a corruption of the remaining point in $A\setminus\{a,b\}$) with a corresponding weight $w_c<w_A/(1+\sigma)=w_A/2$, we can make $m$ a Fermat--Weber point of the new set $\{a,b,c\}$.
    
    From \cite[Theorem~18.37]{BMS:1999} $m$ will be the Fermat--Weber point of $\{a,b,c\}$ as soon as $c\in\RR^2$ and $w_c>0$ satisfy
     \begin{equation} \label{eq:mFWEuc}
        w_a\frac{m-a}{\|m-a\|}+w_b\frac{m-b}{\|m-b\|}+w_c\frac{m-c}{\|m-c\|}=\mathbf 0.
    \end{equation}
    \[\]


Denote $v_{ab}:=w_a\frac{m-a}{\|m-a\|}+w_b\frac{m-b}{\|m-b\|}$,
which cannot be the zero vector, since  $m, a$ and $b$
are not aligned. For the same reason $m-a$ and $m-b$ are independent,
so $0<\|v_{ab}\|<w_a+w_b$.
If we choose $c= m+\rho v_{ab}$ for any $\rho>0$ and $w_c=\|v_{ab}\|$
we will have $w_c \frac{m-c}{\|m-c\|}=w_c \frac{-\rho v_{ab}}{\|\rho v_{ab}\|}=-v_{ab}$, showing that (\ref{eq:mFWEuc}) holds, hence that
$m$ is a Fermat-Weber point of $\{a,b,c\}$ while $w_c<w_a+w_b$, i.e. 
$w_c<w_A/2$.

Since $m$ was arbitrary  outside the line containing $a$ and $b$,
it may be chosen outside any bounded $K$. 
It follows that no bounded $K$ can contain all Fermat--Weber points for all choices of $c$ and any $w_c<w_A/2$. 

This counterexample shows that the Fermat--Weber point for the Euclidean norm in $\RR^2$ is not uniformly robust.
 \end{example}

\begin{example} \label{eg:1dimpGauge}

In the one-dimensional case a gauge is always of the form $\gamma_b(x):=\max((1-b)x,-bx)$ for some $b\in]0,1[$, up to a scaling factor.
One can see that the Fermat--Weber point is then the $b$-quantile which has breakdown point $\min(b,{1-b})$ (see \cite[\textsection 3.2.1]{Staudte+Sheather:1990}), so will lie in $\conv(A)$ whatever change we make to a subsample~$(C,{\bar{w}})$ with ${\bar{w}}_C<\min(b,{1-b})$.
Indeed, after corruption, the fraction of corrupted points smaller than $\min(A)$ or greater than $\max(A)$ is smaller than $\min(b,{1-b})$ so none of them could be a $b$-quantile, leaving the contaminated quantile in $[\min(A),\max(A)]=\conv(A)$.
Therefore, one can select the bounded $K=\conv(A)$ in Theorem~\ref{th:lowerBound}, which is independent of ${\bar{w}}_C$,
showing that all gauges are uniformly robust in dimension 1.
\end{example}

We will show in the following example that separable gauges in $\RR^d$ are also uniformly robust by extending the result from Example~\ref{eg:1dimpGauge}.

\begin{example} \label{eg:sepGauge}
    Let $\gamma$ be a separable gauge, i.e. of the form
  $\gamma(x)=\sum_{i=1}^d \lambda_i\gamma_{b_i}(x_i)$ where each $\gamma_{b_i}$ is a 1-dimensional gauge as in  Example \ref{eg:1dimpGauge} with  $b_i\in]0,1[$ and $\lambda_i>0$ for every $i\in\{1,\dots,d\}$.
    Computing a Fermat--Weber problem on such a separable gauge reduces to computing the $b_i$-quantiles for each coordinate. One can check directly or using \cite[Example~3]{Plastria:1992} that $\gamma$ has skewness  $\sigma=\max_{i}\left(\max(b_i,1-b_i)/\min(b_i,1-b_i)\right)$.
    This implies that, as long as we corrupt only parts of weight smaller than $1/({1+\sigma})=\min_i \min(b_i,1-b_i)$,
    Theorem~\ref{th:lowerBound} applies for each coordinate. Thus, the one-dimensional
    case explained in Example~\ref{eg:1dimpGauge}
     implies that we can set $K=[m_1,M_1]\times\dots\times[m_d,M_d]$, where $m_i=\min_{a\in A}a_i$ and $M_i=\max_{a\in A}a_i$.

\end{example}

In fact the uniform robustness  in Example~\ref{eg:sepGauge} is not an isolated case, but it is caused by the polyhedrality of the gauge, as we show below.
To understand better what is going on, we need to make use of the \emph{elementary convex sets} introduced by Durier and Michelot \cite{Durier+Michelot:85}.

For $p\in\partial B_{\gamma^\circ}$, denote by $N(p)$ the convex cone generated by the exposed face of $B_{\gamma}$ in the direction of $p$, i.e. the face $\partial\gamma^\circ(p)=\{x\in B_{\gamma}: \langle p,x\rangle = 1\}$, hence
 $N(p):=\RR^+\partial\gamma^\circ(p)$. 
For $p\in B_{\gamma^\circ}$ with $\gamma^\circ(p)<1$ we will set $N(p)=\{\mathbf 0\}$.

Figure \ref{fig:TriangleEH} shows on the left the unit ball of the triangular gauge in Figure \ref{fig:ballsSD} (b) dashed. At several typical points of its boundary the subdifferential of this gauge is depicted as dotted vectors: along each triangle's side it is a singleton while at each vertex it is a closed line segment. All these subgradients have dual unit length, so they define convex cones of type $N(p)$; e.g. $N(p_1)$ is the closure of the grey area and  $N(p)$ is the closed vertical half-line issued from the origin $\0$ for any $p\in]p_1,p_2[$, the relative interior of the subgradient at the top vertex.

\begin{definition}
Consider $\pi=(p_a)_{a\in A}$ a finite family of points from the unit ball $B_{\gamma^\circ}$ of the dual gauge and let $C_\pi=\bigcap_{a\in A}({a+N(p_a)})$.
An \emph{elementary convex set} for $A$  is a non-empty set of the form $C_{\pi}$.
\end{definition}

If $\gamma^\circ(p_a)<1$ for some $a\in A$, then $C_\pi=\{a\}$ or it is empty; in particular, every point of $A$ forms an elementary convex set. All elementary convex sets for $A$ cover $\RR^d$, 
while their relative interiors partition $\RR^d$, see \cite{Durier:90}. 

\begin{remark} \label{rmk:defFamElemConvSet}
    We want to emphasize that an elementary convex set $S$ for $A$ can be defined by multiple families of points from $B_{\gamma^\circ}$.
    We use that \[x\in a+N(p_a)=a+\RR^+\partial\gamma^\circ(p_a)\text{ if and only if }\langle p_a,x-a\rangle=\gamma(x-a).\] 
    But we also know that $\langle p_a,x-a\rangle=\gamma(x-a)$ is equivalent to $p_a\in\partial\gamma(x-a)$, so
    \begin{equation}\label{eq:xincone}
        x\in a+N(p_a)\ \text{ if and only if }\ p_a\in\partial\gamma(x-a).
    \end{equation}

    Any family $\pi=(p_a)_{a\in A}$ with $p_a$ in the relative interior of $\partial\gamma(x-a)$ for each $a\in A$, where $x$ is a point in the relative interior of $S$, satisfies $S=C_\pi$; see Appendix~\ref{sec:ElemConvSets}.
\end{remark}

Figure \ref{fig:TriangleEH} shows the elementary convex sets for the set $A=\{a,b,c,d,e\}$. 
Each closed region, line segment and intersection point is an  elementary convex set.
E.g. the thick vertical closed segment immediately above $a$ equals $C_\pi$ for 
the family $\pi=(p_4,p_3,p_1,p_3,p_3)$, in which $p_4$ may be replaced by any $p\in]p_1,p_2[$. Also the thick closed half-line $L$ issued from $c$ coming from $e$ equals $C_\pi$ for any $\pi=(p_2,p_3,p_c,p_2,p_e)$ with both $p_c$ and $p_e$ in $]p_2,p_3[$; other choices are possible, e.g. $(p_c,p_e)=(p_3,p_2)$.

\begin{figure}
\centering
   \begin{tikzpicture}[scale=0.3];
  \begin{scope}
    \path (0,8) coordinate (top);
    \path (-30:8) coordinate (right);
    \path (210:8) coordinate (left);
    \draw[dashed] (top) -- (right) -- (left) -- cycle;
   \draw[fill] (0,0) coordinate (o) circle(0.2) node[below] {$\0$}; 

\draw[very thick]     (o) -- ($(o)!1.5!(top)$) coordinate (T); 
\draw[]     (o) -- ($(o)!1.5!(right)$) coordinate (R); 
\draw[]     (o) -- ($(o)!1.5!(left)$) coordinate (LL); 

    \path ($(top)!0.3!(left)$) coordinate (tl);
    \draw[dotted,-stealth] (tl) -- ++(150:3) node[below]{$p_1$};
    \draw[dotted,-stealth] (top)--++(150:3) coordinate (topl) node[below]{$p_1$};
   \draw[dotted,-stealth] (left)--++(150:3) coordinate (leftl) node[below]{$p_1$};
    \path ($(top)!0.5!(right)$) coordinate (tr);
    \draw[dotted,-stealth] (tr) -- ++(30:3)  node[below]{$p_2$};
    \draw[dotted,-stealth] (top)--++(30:3) coordinate (topr) node[below]{$p_2$};
    \draw[dotted,-stealth] (right)--++(30:3) coordinate (rightr) node[below]{$p_2$};
    \draw[dotted,-stealth] (top)--($(topl)!0.75!(topr)$) coordinate (topm) node[above]{$p_4$};
    \draw[dotted] (topl) -- (topm) -- (topr);
    \path ($(left)!0.2!(right)$) coordinate (lr);
    \draw[dotted,-stealth] (lr) -- ++(-90:3)  node[right]{$p_3$};
    \draw[dotted,-stealth] (left) -- ++(-90:3)   coordinate (leftr) node[right]{$p_3$};
    \draw[dotted,-stealth] (right) -- ++(-90:3)  coordinate (rightl)  node[right]{$p_3$};
    \draw[dotted] (leftl) -- (leftr) (rightl) -- (rightr);

    \fill[lightgray, fill opacity=0.5]  (T) -- (o) -- (LL) -- (-10,12) -- (T); 
    \draw[] (-7,2) node[fill=white] {\small $N(p_1)$};
    \draw[] (0,3) node[rotate=90,fill=white,draw] {\small $N(p_4)$};
  \end{scope}
\begin{scope}[shift={(20,0)}]
    \path[clip] (-6,-5) rectangle (10,10);
    
\foreach \h/\v/\n in {1/0/a, 2/5/b, 5/1/c, -3/4/d, 1.6/3/e} 
{ 
\path (\h,\v) coordinate (\n);
\begin{scope}[shift={(\n)}]
     \draw[fill] (0,0)  circle(0.2) node[below] {$\n$}; 
     \draw (0,0) -- (90:20)  (0,0) -- (210:20)  (0,0) -- (-30:20);
          \end{scope}
};

\path[name path=leftb] (b) -- ++(210:20);
\path[name path=rightd] (d) -- ++(-30:20);
\path[name intersections={of=leftb and rightd} ] (intersection-1) coordinate (bd);

\path[name path=rightb] (b) -- ++(-30:20);
\path[name path=topc] (c) -- ++ (90:20); 
\path[name intersections={of=rightb and topc} ] (intersection-1) coordinate (bc);

\path[name path=righta] (a) -- ++(-30:20);
\path[name path=leftc] (c) -- ++(210:20);
\path[name intersections={of=righta and leftc} ] (intersection-1) coordinate (ac);

\path[name path=topa] (a) -- ++ (90:20); 
\path[name path=rightd] (d) -- ++(-30:20);
\path[name intersections={of=topa and rightd} ] (intersection-1) coordinate (ad);

\fill[lightgray, fill opacity=0.5] (bd)--(b)--(bc)--(c)--(ac)--(a)--(ad)--cycle;

  \draw[thick](a)--(ad);
  \draw[thick](c)-- ++(-30:20) node[pos=0.1,below]{\tiny $L$} ;
  \end{scope}
  
  \end{tikzpicture}
 \caption{On the left the normal cones and on the right the elementary hull of a 5-point set for the triangular unit ball (b) of Figure~\ref{fig:ballsSD} }\label{fig:TriangleEH}
\end{figure}

\begin{definition} \label{def:EH} 
 The \emph{elementary hull} $\EH_{\gamma}(A)$ of $A$ is the union of all bounded elementary convex sets for $A$.
\end{definition}

The elementary hull of the 5-point set $A$ in Figure \ref{fig:TriangleEH}
is shown at right as a grey area (with its boundary and a segment with one endpoint at $d$).
It is worth noting that in this specific case, the elementary hull is equivalent to the tropical convex hull as defined in \cite{Develin+Sturmfels:2004}, as proven in Theorem~15 of the aforementioned paper.

As another example, consider a separable gauge defined by 
$\gamma(x)=\sum_{i=1}\gamma_{b_i}(x_i)$ as in Example \ref{eg:sepGauge}.
Then $\EH_{\gamma}(A)$ is exactly the set $K$ defined there.

Note that when $D\subset A$ each elementary convex set for $A$ is a subset of some elementary convex set for $D$. 
If the former is unbounded, so will be the latter. It follows that $\RR^d\setminus \EH_{\gamma}(A)\subset\RR^d\setminus \EH_{\gamma}(D)$, hence $\EH_{\gamma}(D)\subset\EH_{\gamma}(A)$.

Since some elementary convex sets are unbounded, it is useful to consider their recession cones.
For a non-empty convex set $X$, its \emph{recession cone} is defined as:
\[\rec(X):=\{u\in\RR^d:x+\lambda u\in X\text{ for every }\lambda\geq 0\text{ and }x\in X\}.\]
We will use later the fact that $\rec(X)=\{\0\}$ if and only if $X$ is bounded whenever $X$ is a non-empty closed convex set~\cite[Theorem~8.4]{Rockafellar:70}.

\begin{proposition} \label{prop:elemSet}
    Let $x$ be a Fermat--Weber point for a  sample $(A,w)$ in which a subsample $(C,{\bar{w}})$ with ${\bar{w}}_C/w_A<1/(1+\sigma)$ was corrupted.
    Then $x\in\EH_{\gamma}(D)$, where $(D,w')=(A,w)-(C,{\bar{w}})$.
\end{proposition}

\begin{proof}
    Assume that there exists an unbounded elementary convex set $S$ for $D$ 
    containing $x$ in its relative interior.
    Since $S$ is an unbounded convex set, there exists a non-zero vector $r$ in its recession cone and we can choose it with $\gamma(r)=1$.
    This means that for every $s\in S$ we also have $s+\lambda r\in S$ for all $\lambda\geq 0$. 
    Consider then  a family $\pi=(p_d)_{d\in D}$ such that $p_d\in\partial\gamma(x-d)$ for every $d\in D$ defining $S=C_{\pi}$ and $x\in C_{\pi}$. Remark~\ref{rmk:defFamElemConvSet} implies 
    \[\frac{\frac{x-d}{\lambda}+ r}{\gamma(\frac{x-d}{\lambda}+ r)}=
    \frac{x-d+\lambda r}{\gamma(x-d+\lambda r)}\in\partial\gamma^\circ(p_d)\ \text{ for every }\lambda\geq 0 \text{ and }d\in D.\]
    Letting $\lambda\to\infty$ we obtain
    \(r\in\partial\gamma^\circ(p_d)\)  for every \(d\in D\)
    because subdifferentials are closed sets.
    Hence \(\langle p_d,r\rangle=\gamma^\circ(p_d)=1\) since 
    $p_d\in\partial\gamma(x-d)$.
    Thus, for every $p\in\sum_{d\in D}w'_d\partial\gamma(x-d)$, we have $\langle p,r\rangle=\sum_{d\in D}w'_d\langle p_d,r\rangle=\sum_{d\in D}w'_d\gamma^\circ(p_d)=w'_D$.

    Let $(C_c,{\bar{w}}')$ be the corruption of $(C,{\bar{w}})$.
    Since we assumed $x$ to be a Fermat--Weber point of $(D,w')+(C_c,{\bar{w}}')$, there exists
    \[q\in\sum_{c\in C_c}{\bar{w}}'_c\partial\gamma(x-c)\ \text{ such that }-q\in\sum_{d\in D}w'_d\partial\gamma(x-d).\]
    In particular, $\langle -q,r\rangle=w'_D$.
    But \[\langle -q,r\rangle\leq\gamma^\circ(-q)\gamma(r)\leq\sigma\gamma^\circ(q)\leq\sigma\cdot 
    {\bar{w}}_C,\] where the last inequality comes from Lemma~\ref{lem:part1}.
    Therefore, \[\sigma/(1+\sigma)<1-{\bar{w}}_C/w_A=w'_D/w_A=\langle -q,r\rangle/w_A\leq\sigma\cdot {\bar{w}}_C/w_A<\sigma/(1+\sigma),\] which is a contradiction.

   Since the relative interiors of elementary convex sets partition
   $\RR^d$, $x$ must belong to one of these. We just showed that this cannot be an unbounded one, so we conclude that $x$ 
    must belong to $\EH_\gamma(D)$.
\end{proof}

The previous proposition does not improve Theorem~\ref{th:lowerBound} when $\gamma$ is strictly convex, since this case entails $\EH_{\gamma}(D)=\RR^d$ when $D$ consists of at least $d+1$ points in general position.
However, a better result can be obtained for the polyhedral case.

\begin{corollary} \label{cor:polyhedral}
 The Fermat--Weber point is uniformly robust for any polyhedral gauge, i.e. one can select a bounded $K$ in Theorem~\ref{th:lowerBound} independent of ${\bar{w}}_C$ when $\gamma$ is polyhedral.
\end{corollary}

\begin{proof}
    From Proposition~\ref{prop:elemSet}, any corrupted Fermat--Weber point belongs to some $\EH_{\gamma}(D)$, which is a subset of $\EH_{\gamma}(A)$. Each elementary convex set for $A$ is an intersection of a finite number of cones (one for each $a\in A$), each of which is a cone over one of the finitely many faces of the polyhedral unit ball of $\gamma$. Thus  there are only finitely many distinct elementary convex sets for $A$.
  Accordingly, the elementary hull $\EH_{\gamma}(A)$ is a union of finitely many bounded sets, and is therefore bounded.
    As a consequence, we can consider $K=\EH_{\gamma}(A)$ in Theorem~\ref{th:lowerBound}.
\end{proof}

As the Fermat--Weber point may not be unique, Corollary~\ref{cor:polyhedral} should be understood in terms of the set of all Fermat--Weber points.
The uniform robustness guarantees that for any $1/(1+\sigma)$-corruption, the entire set of Fermat--Weber points of the corrupted sample remains contained within a fixed bounded set $K$ (determined by the original uncorrupted sample), regardless of the specific $1/(1+\sigma)$-corruption; see Remark~\ref{rmk:uniformRobust-nonunique}.


\section{Uniform robustness for norms} \label{sec:unifRobustNorm}

We now will rather analyze the position of the Fermat--Weber point of 
a sample $(D,w)$ that is contaminated by $(C,{\bar{w}})$.
We showed in Proposition~\ref{prop:elemSet} that $\EH_\gamma(D)$ contains all  Fermat--Weber points for the contaminated sample $(D,w)+(C,{\bar{w}})$ 
of full weight $w_D+{\bar{w}}_C$ as soon as ${\bar{w}}_C/(w_D+{\bar{w}}_C)<1/(1+\sigma)$, or equivalently $\sigma{\bar{w}}_C<w_D$.
Below we will prove the converse for the case when $\gamma$ is a norm, i.e. $\sigma=1$.
But, firstly, we need two lemmas, which are also valid for gauges.

\begin{lemma} \label{lem:linear}
    Let $(D,w)$ be a weighted set in $\RR^d$.
    For any gauge $\gamma$ we then have $\gamma\left(\sum_{d\in D}w_d d\right)=\sum_{d\in D}w_d \gamma(d)$ if and only if $\bigcap_{d\in D}\partial\gamma(d)\neq\emptyset$.
\end{lemma}

\begin{proof}
    Firstly, assume that $\bigcap_{d\in D}\partial\gamma(d)\neq\emptyset$.
    We will use repeatedly \cite[Lemma~8]{Plastria-pasting-gauges1}, which states (among other) that $\partial\gamma(x)\cap\partial\gamma(y)=\partial\gamma(z)$ for every $z\in]x,y[$ whenever this intersection is non-empty, where $x,y$ are distinct points in $\RR^d$.
    Using also the fact that subdifferentials of gauges are invariant under positive scalings, the corresponding statement for larger intersections gives \[\bigcap_{d\in D}\partial\gamma(d)=\partial\gamma\left(\frac{1}{w_D}\sum_{d\in D}w_d d\right)=\partial\gamma\left(\sum_{d\in D}w_d d\right).\]
    Let $r$ be a point from the latter set.
    Then $\gamma(d)=\langle r,d\rangle$ for every $d\in D$ and $\gamma\left(\sum_{d\in D}w_d d\right)=\left\langle r,\sum_{d\in D}w_d d\right\rangle=\sum_{d\in D}w_d \left\langle r,d\right\rangle=\sum_{d\in D}w_d \gamma(d)$.

    Secondly, we consider the case when $\bigcap_{d\in D}\partial\gamma(d)=\emptyset$ and let $q\in\partial\gamma\left(\sum_{d\in D}w_d d\right)$.
    Then, using the generalized Cauchy--Schwarz inequality, we obtain
    \[\gamma\left(\sum_{d\in D}w_d d\right)=\left\langle q,\sum_{d\in D}w_d d\right\rangle=\sum_{d\in D}w_d\langle q,d\rangle<\sum_{d\in D}w_d\gamma(d).\]
    The inequality is strict because the weights are positive, 
    $q$ is a subgradient of $\gamma$ so $\gamma^\circ(q)\leq 1$, and by the generalized Cauchy--Schwarz inequality we have, for all $d\in D$, that
    \[\langle q,d\rangle\leq\gamma^\circ(q)\gamma(d)\leq\gamma(d),\]
    while there is at least one $d\in D$ for which $q\not\in\partial\gamma(d)$, implying that 
    the first of these inequalities is strict. 
\end{proof}

\begin{lemma} \label{lem:rec_Cpi}
    Let $C_\pi$ be an elementary convex set with respect to $D$ (and the gauge~$\gamma$) induced by $(p_d)_{d\in D}$ and such that $C_\pi\neq\{d\}$ for every $d\in D$.
    Then $C_\pi$ is bounded if and only if $\bigcap_{d\in D}\partial\gamma^\circ(p_d)$ is empty.
\end{lemma}

\begin{proof}
    For $p_d\in\partial B_{\gamma^\circ}$ we have $N(p_d)=\RR^+\partial\gamma^\circ(p_d)$, i.e. the cone generated by the exposed face $\partial\gamma^\circ(p_d)$ of $B_\gamma$.
    This allows us to obtain \[\bigcap_{d\in D} N(p_d)=\{\0\}\cup\RR^+\bigcap_{d\in D}\partial\gamma^\circ(p_d)\] because $u\in N(p_d)$ if and only if $u=\0$ or $u/\gamma(u)\in\partial\gamma^\circ(p_d)$.
    We used that $C_\pi\neq\{d\}$ in order to write $N(p_d)=\RR^+\partial^\circ(p_d)$.
    If it were not the case, $C_\pi$ could have been induced by some $p_d\in B_{\gamma^\circ}$ with $\gamma^\circ(p_d)<1$.
    
    What is more, we obtain
    \begin{eqnarray*}
        \rec (C_\pi)&=&\rec\left(\bigcap_{d\in D}(d+N(p_d))\right)
                    = \bigcap_{d\in D}\rec(d+N(p_d))\\
                    &=& \bigcap_{d\in D} N(p_d)
                    = \{\0\}\cup\RR^+\bigcap_{d\in D}\partial\gamma^\circ(p_d)
    \end{eqnarray*}
    whenever $C_\pi\neq\{d\}$ for every $d\in D$.
    We could interchange the intersection with taking recession cones due to \cite[Corollary~8.3.3]{Rockafellar:70} since the sets $d+N(p_d)$ are closed and convex and their intersection, $C_\pi$, is non-empty.

    Note now that it is not possible to have $\bigcap_{d\in D}\partial\gamma^\circ(p_d)=\{\0\}$ since $\partial\gamma^\circ(p_d)\subset\partial B_\gamma$ when $C_\pi\neq\{d\}$. Accordingly, $C_\pi$ is bounded if and only if
    $\rec (C_\pi)=\{\0\}$ if and only if $\bigcap_{d\in D}\partial\gamma^\circ(p_d)$ is empty.
\end{proof}

\begin{theorem} \label{th:EHnorms}
    If $\gamma$ is a norm, then for every $m\in\EH_\gamma(D)$ there exists $c\in\RR^d$ and ${\bar{w}}_c\in ]0,w_D[$ such that $m$ is a Fermat--Weber point of the contaminated set $(D,w)+(c,{\bar{w}}_c)$.
\end{theorem}

\begin{proof}
    If $m\in D$, we can set $c=m$ and ${\bar{w}}_c=w_{D\setminus \{m\}}<w_{D}$.
    For the contaminated set $(D,w)+(c,{\bar{w}}_c)$ the point $m=c$ then has weight ${\bar{w}}_c+w_m$ which is higher than the sum of weights of all other points $w_{D\setminus\{m\}}$, so a majority weight. Witzgall's majority theorem \cite{Plastria2020, Witzgall:TechnicalReport:1964} (for more details see Theorem~\ref{th:majority} below) then ensures that  $m$ is a Fermat--Weber point of the contaminated set $(D,w)+(c,{\bar{w}}_c)$.

    From now on, assume that $m\in \EH_\gamma(D)\setminus D$.
    Let $C_\pi$ be the smallest elementary convex set for $D$ containing $m$ and  let $C_\pi$ be determined by $p_d\in\partial\gamma(m-d)\subset\partial B_{\gamma^\circ}$ for every $d\in D$.
    Since $m\in \EH_\gamma(D)$ it lies in some bounded elementary convex set for $D$, which hence contains $C_\pi$, so this latter must be bounded.
        Consequently, Lemma~\ref{lem:rec_Cpi} implies $\bigcap_{d\in D}\partial\gamma^\circ(p_d)=\emptyset$.

    If $m$ is a Fermat--Weber point of $(D,w)$, then we could simply consider $c=m$ and ${\bar{w}}_c=w_D/2$.
    Indeed, adding $c=m$ with a positive weight ${\bar{w}}_c$ to $(D,w)$ would make $m$ the unique Fermat--Weber point of $(D,w)+(c,{\bar{w}}_c)$.

    If $m$ is not a Fermat--Weber point of $(D,w)$, then we consider ${\bar{w}}_c:=\gamma^\circ\left(-\sum_{d\in D}w_d p_d\right)$.
    This value is strictly positive since the contrary would imply $\0=\sum_{d\in D}w_dp_d\in\sum_{d\in D}w_d\partial\gamma(m-d)$ showing that $m$ is actually a Fermat--Weber point of $(D,w)$.
    Since $\gamma^\circ$ is also symmetric, we can write ${\bar{w}}_c=\gamma^\circ\left(\sum_{d\in D}w_d p_d\right)$.
    What is more, Lemma~\ref{lem:linear} implies that \[{\bar{w}}_c<\sum_{d\in D}w_d\gamma^\circ(p_d)=w_D\] because $\bigcap_{d\in D}\partial\gamma^\circ(p_d)=\emptyset$.
    
    Set $u:=-\frac{1}{{\bar{w}}_c}\sum_{d\in D}w_d p_d$, choose $q\in\partial\gamma^\circ(u)$ and define $c=m-q$.
    Then \[\sum_{d\in D}w_d p_d+{\bar{w}}_c u=\0.\]
    Also, since $\gamma^\circ(u)=1$, it follows that $u\in\partial\gamma(q)=\partial\gamma(m-c)$.
    Given that $p_d\in\partial\gamma(m-d)$, this shows that
    \[\0\in\sum_{d\in D}w_d\partial\gamma(m-d)+{\bar{w}}_c\partial\gamma(m-c),\]
    so $m$ is a Fermat--Weber point of the weighted set $(D,w)+(c,{\bar{w}}_c)$.
\end{proof}

\begin{corollary} \label{cor:charUnifRobustNorm}
    A norm $\gamma$ is uniformly robust if and only if $\EH_\gamma(A)$ is bounded for every finite set $A\subset\RR^d$.
\end{corollary}

\begin{proof}
    Firstly, assume that $\gamma$ is uniformly robust and let $A\subset\RR^d$ be finite.
    From the definition of uniform robustness, there exists a bounded set $K(A)$ such that every corruption of $A$ (with all weights equal to 1) with less than half of the sample will result in a Fermat--Weber point inside $K(A)$.
    Note that the exact value of the weights is irrelevant since gauges are equivariant under scalings; the only relevant thing is that the weights are all equal.
    Now consider $m\in\EH_\gamma(A)$ arbitrary.
    From Theorem~\ref{th:EHnorms}, there exists $c\in\RR^d$ and ${\bar{w}}_c<|A|$ such that $m$ is Fermat--Weber point for the set $A+(c,{\bar{w}}_c)$.
    But this can be seen as a corruption of the set $A$ where all the weights are $1+{\bar{w}}_c/|A|$.
    Thus, $m\in K(A)$.
    Since $m$ was chosen arbitrarily, we have $\EH_\gamma(A)\subset K(A)$
    and so $\EH_\gamma(A)$ is bounded.

    Conversely, assume that $\EH_\gamma(A)$ is bounded for every finite $A\subset\RR^d$.
    From Theorem~\ref{th:EHnorms}, every Fermat--Weber point of a corruption of $(A,w)$ must belong to $\EH_\gamma(D)$ for some subset $D$ of $A$.
    Hence, every corresponding corrupted estimator will belong to $\bigcup_{D\subset A}\EH_\gamma(D)\subset \EH_\gamma(A)$ which is bounded by assumption.
\end{proof}

We already remarked that $\EH_\gamma(A)$ is generally unbounded when $\gamma$ is strictly convex.
In fact, we will show that we need this property only locally.

\begin{definition} \label{def:localStrictConv}
    We call a gauge $\gamma$ \emph{locally strictly convex} at $p\in\partial B_\gamma$ if for every $q\in\partial B_\gamma\setminus\{p\}$ and $\lambda\in]0,1[$ we have $\gamma\left((1-\lambda)p+\lambda q\right)<1$.
\end{definition}

The definition is equivalent to the property that the open segment $]p,q[$ belongs to the interior of $B_\gamma$ for every boundary point $q$ which is different from $p$.
Due to the convexity of $\gamma$ and Jensen’s inequality, it is sufficient to check this property only for $q$ belonging to a neighbourhood of $p$.

\begin{proposition} \label{prop:equivLocStrConv}
    A gauge $\gamma$ is locally strictly convex at $p\in\partial B_\gamma$ if and only if $\partial\gamma^\circ(v)=\{p\}$ for each  $v\in\partial\gamma(p)$, or, equivalently,
    $\gamma^\circ$ is differentiable at every point of $\partial\gamma(p)$.
\end{proposition}

\begin{proof}
    Firstly, assume that $\gamma$ is locally strictly convex at $p\in\partial B_\gamma$ and consider any $v\in\partial\gamma(p)$.
        Assume there exists some $q\in\partial\gamma^\circ(v)$ with $q\neq p$, then $\gamma((p+q)/2)<1$ as $\gamma$ is locally strictly convex at $p$.
    This would imply that ${(p+q)/2}\notin\partial\gamma^\circ(v)$, although $p$ and $q$ are points of $\partial\gamma^\circ(v)$, contradicting the convexity of subdifferentials.
    Consequently, we must have $\partial\gamma^\circ(v)=\{p\}$, 
     which is equivalent to $\gamma^\circ$ being differentiable at $v$.

    Conversely, assume that $\gamma$ is not locally strictly convex at $p$.
    Thus, there is $q\in\partial B_\gamma\setminus\{p\}$ 
    such that the segment $[p,q]$ is contained in the boundary $\partial B_\gamma$.
    Due to the convexity of $B_\gamma$, there exists a vector $u\in\partial B_{\gamma^\circ}$ which is outer normal to a supporting hyperplane of $B_\gamma$ containing $p$ and $q$.
    In particular, $u\in\partial\gamma(p)\cap\partial\gamma(q)$ and  $\{p\}\neq[p,q]\subset\partial\gamma^\circ(u)$.
    We infer that $\gamma^\circ$ is not differentiable at $u\in\partial\gamma(p)$.
\end{proof}

\begin{proposition} \label{prop:unboundedEHLocStrConv}
    If $\gamma$ is a locally strictly convex gauge at $p\in\partial B_\gamma$ and $a,b\in\RR^d$ such that $b-a$ is not parallel to $p$, then $\EH_\gamma(\{a,b\})$ is unbounded.
\end{proposition}

\begin{proof}
    Let $u_\lambda\in\partial\gamma(a+\lambda p-b)\subset\partial B_{\gamma^\circ}$ for an arbitrary $\lambda\geq 0$.
    If $u_\lambda\in\partial\gamma(p)$, then $\partial\gamma^\circ(u_\lambda)=\{p\}$ from Proposition~\ref{prop:equivLocStrConv} since $\gamma$ is locally stricly convex at $p\in\partial B_\gamma$.
    But by equation (\ref{eq:dualsubdiff}) \[(a+\lambda p-b)/\gamma(a+\lambda p-b)\in\partial\gamma^\circ(u_\lambda)\] so we would have \[(a+\lambda p-b)/\gamma(a+\lambda p-b)=p,\]
    implying that $b-a$ is parallel to $p$, contradicting our assumptions.
    Therefore, $u_\lambda\notin\partial\gamma(p)$ and hence 
    $p\notin \partial\gamma^\circ(u_\lambda)$, showing that $p$
    does not lie in the recession cone of $N(u_\lambda)$.
    
    Consider now any $u\in\partial\gamma(p)$, then $a+\lambda p \in a+N(u)=a+\RR^+\partial\gamma^\circ(u)$, which is a half-line, by  Proposition~\ref{prop:equivLocStrConv}, that is parallel to $p$.
    The intersection of the aforementioned half-line with the cone $b+N(u_\lambda)$ is therefore a bounded segment that contains $a+\lambda p$.
    Consequently, $a+\lambda p \in (a+N(u))\cap(b+N(u_\lambda))$, which is a bounded elementary convex set, so  $a+\lambda p\in \EH_\gamma(\{a,b\})$.
    
    $\lambda\geq 0$ being arbitrary, this shows that $a+\RR^+p\subset\EH_\gamma(\{a,b\})$.
\end{proof}

A direct consequence of Corollary~\ref{cor:charUnifRobustNorm} and Proposition~\ref{prop:unboundedEHLocStrConv} is the following.

\begin{corollary} \label{cor:localStricConvNorms}
    Norms that are locally strictly convex at a point are not uniformly robust.
\end{corollary}

The cases considered in Corollaries~\ref{cor:polyhedral} and \ref{cor:localStricConvNorms} do not cover all possible norms, even  in two dimensions, as the following counterexample shows. 
Afterwards we will show that in dimension 2 the remaining gap  may be closed,
thus obtaining a full characterization of planar uniformly robust norms.

\begin{example} \label{ex:strictlyLocalConvSetFromOneSide}
    Take $B_\gamma$ as the convex hull of the points $\pm(1,0)$, $\pm(0,1)$, $\pm( 1/\sqrt{2},- 1/\sqrt{2})$, and the points $\pm v_n$ where $v_n:=\left(\cos(\pi/3n),\sin(\pi/3n)\right)$ for $n\geq 1$, as displayed in Figure~\ref{fig:nonpolyhNotStrictLocalConvexSet}.
    The set $B_\gamma$ is described as the convex hull of a compact set, which we will denote by $V$.
    Indeed, $V$ is both bounded and contains its accumulation points, namely $\pm (1, 0)$.
    Consequently, $B_\gamma$ is a compact convex set containing $\0$ in its relative interior, so it is indeed the unit ball of a norm.

    On the other hand, there is a small neighbourhood of every point of $B_\gamma\setminus\{(\pm 1,0)\}$ such that the restriction of $B_\gamma$ to the neighbourhood coincides with a polyhedral set.
    This is due to the fact that $(\pm 1,0)$ are the only non-isolated accumulation points of the vertex set~$V$.
    But $B_\gamma$ is not polyhedral since there are infinitely many exposed points.

    However, $B_\gamma$ is not locally strictly convex at the points $\pm (1,0)$ because each of them is contained in a 1-dimensional exposed face: the segment between $(1,0)$ and $(1/\sqrt{2},-1/\sqrt{2})$ and its reflection about the origin, respectively.
    Nevertheless, the condition from Definition~\ref{def:localStrictConv} would hold for $(1,0)$ if we restrict to points $q$ having $q_2>0$.
    For this reason, we could say that $B_\gamma$ is ``strictly locally convex from the positive side'' at $(1,0)$.
\end{example}

\begin{figure}
    \centering
    \scalebox{.6} {
    
        \begin{tikzpicture}[scale=5]
    
        \node[circle, draw=black, fill=black, inner sep=1.2, label=below left:{\large $\0$}] (O) at (0,0) {};
    
        \coordinate (Nord) at (0,1) {};
        \coordinate (Sud) at (0,-1) {};
        \coordinate (Est) at (1,0) {};
        \coordinate (Vest) at (-1,0) {};
    
        \coordinate (SE) at ({cos(-45)},{sin(-45)}) {};
        \coordinate (NV) at ({cos(135)},{sin(135)}) {};
    
        \coordinate (P1) at ({cos(60)}, {sin(60)})  {};
        \coordinate (P2) at ({cos(30)}, {sin(30)}) {};
        \coordinate (P3) at ({cos(20)}, {sin(20)}) {};
        \coordinate (P4) at ({cos(15)}, {sin(15)}) {};
        \coordinate (P5) at ({cos(10)}, {sin(10)}) {};
        \coordinate (P6) at ({cos(5)}, {sin(5)}) {};
        \coordinate (P7) at ({cos(2)}, {sin(2)}) {};
    
        \foreach \x in {1,...,7} {
            \coordinate (Q\x) at ($-1*(P\x)$);
        }
    
        \draw (Nord) -- (P1) node[fill=black,inner sep=1,label=right:{\large $v_1$}]{}
        -- (P2) node[fill=black,inner sep=1,label=right:{\large $v_2$}]{}
        -- (P3) node[fill=black,inner sep=1,label=right:{\large $v_3$}]{}
        -- (P4) node[fill=black,inner sep=1,label=right:{\large $v_4$}]{}
        -- (P5) node[fill=black,inner sep=1,label=right:{\large $v_5$}]{}
        -- (P6) node[fill=black,inner sep=1,label=right:{\large $v_6$}]{}
        -- (P7) node[fill=black,inner sep=1,label=right:{\tiny$\vdots$}]{}
        -- (Est) node[fill=white,draw=black,inner sep=1.5,label=below right:{\large $u$}]{}
        -- (SE) -- (Sud) -- (Q1) -- (Q2) -- (Q3) -- (Q4) -- (Q5) -- (Q6) -- (Q7) -- 
        (Vest) -- (NV) -- cycle;
    
        \end{tikzpicture}

    }
    
    \caption{Non-polyhedral convex set described in Example~\ref{ex:strictlyLocalConvSetFromOneSide} which is not strictly locally convex at any point. The point $u$ is an exposed point that is the limit of the exposed points $v_n$, but is also contained in a larger exposed face of the polyhedral set}
    \label{fig:nonpolyhNotStrictLocalConvexSet}
\end{figure}

\begin{proposition} \label{prop:nonPolyhGauge2Dim}
    Let $\gamma$ be a non-polyhedral gauge on $\RR^2$.
    Then there exist ${a,b\in\RR^2}$ such that $\EH_\gamma(\{a,b\})$ is unbounded.
\end{proposition}

\begin{proof}
    Since $\gamma$ is non-polyhedral, $B_\gamma$ has infinitely many exposed points.
    These points are contained in a compact set, so they have an accumulation point, $u$.
    The point $u$ belongs to $\partial B_\gamma$ because the boundary is a compact set.
    
    We remark that linear transformations map elementary convex sets with respect to a set $A$ to elementary convex sets with respect to the transformed $A$ in a new normed space.
    For this reason, we make a few simplified assumptions about~$u$.
    More precisely, we assume that $u=(1,0)$, and that $u$ is a limit of exposed points $x$ with $x_2>0$.
    These assumptions could be obtained from a rotation, scaling and, if necessary, a reflection about the $x$-axis.
    Such a situation is depicted in Figure~\ref{fig:nonpolyhNotStrictLocalConvexSet}.

    We may then conclude that for every $a\in\partial B_\gamma\cap\{x\in\RR^2:x_2>0\}$ and $\lambda\in]0,1[$, we have $\gamma((1-\lambda)u+\lambda a)<1$.
    Indeed, if it were not the case, there would exist a point 
    \(q\in]a,u[\) with $\gamma(a)=\gamma(q)=\gamma(u)=1$, which by \cite[Lemma~8]{Plastria-pasting-gauges1} implies
    \([q,u]\subset\partial B_\gamma\cap{\{x\in\RR^2:x_2>0\}}\);
    thus at every boundary point $x$ of $B_\gamma$ with $x_2>0$ that is closer to $u$ than $q$ we have the same unique supporting line
 which contains $[q,u]$, so its intersection with $B_\gamma$ is not equal to $\{x\}$, hence $x$ is not exposed; it would follow that
$u$ would not be a limit of exposed points $x$ with $x_2>0$, contradicting how we defined~$u$.
    
    Consequently, any cone $N(p)$ with $p\in\partial\gamma(u)$ is contained in the half-space $\{{x\in\RR^2}:x_2\leq 0\}$.
    For example, in Figure~\ref{fig:nonpolyhNotStrictLocalConvexSet} there are two possible cones $N(p)$: the ray starting from the origin which contains $u$ and the cone generated by the 1-dimensional face containing $u$.
    In both cases, the cones lie in the half-space of points with non-positive $x_2$-coordinates.

    We consider $a$ arbitrarily and $b$ any point with $b_2>a_2$.
    The set \[\{(b+\lambda u-a)/\gamma(b+\lambda u-a):\lambda\geq 0\}\] is a connected subset of $\partial B_\gamma\cap\{x\in\RR^2:x_2>0\}$ having $u$ as an accumulation point.
    In particular, it contains a sequence of exposed points of $B_\gamma$ converging to~$u$.
    That is, there exists a sequence $(\lambda_n)$ of positive numbers such that $\lambda_n\to\infty$ and \[v_n:=({b+\lambda_n u-a})/\gamma({b+\lambda_n u-a})\] is an exposed point of $B_\gamma$.
    Refer to Figure~\ref{fig:nonpolyhNotStrictLocalConvexSet} for an example.

    \begin{figure}
    \centering
    
    \begin{tikzpicture}[scale=1.2]

    \coordinate (Nord) at (0,1) {};
    \coordinate (Sud) at (0,-1) {};
    \coordinate (Est) at (1,0) {};
    \coordinate (Vest) at (-1,0) {};

    \node[circle,fill=black!70!white,draw=black,inner sep=1.2,label=left:{$a$}] (a) at (0,0) {};

    \coordinate (SE) at ({cos(-45)},{sin(-45)}) {};
    \coordinate (NV) at ({cos(135)},{sin(135)}) {};

    \coordinate (P1) at ({cos(60)}, {sin(60)})  {};
    \coordinate (P2) at ({cos(30)}, {sin(30)}) {};
    \coordinate (P3) at ({cos(20)}, {sin(20)}) {};
    \coordinate (P4) at ({cos(15)}, {sin(15)}) {};
    \coordinate (P5) at ({cos(10)}, {sin(10)}) {};
    \coordinate (P6) at ({cos(5)}, {sin(5)}) {};
    \coordinate (P7) at ({cos(2)}, {sin(2)}) {};

    \foreach \x in {1,...,7} {
        \coordinate (Q\x) at ($-1*(P\x)$);
    }

    \draw (Nord) -- (P1)
    -- (P2) -- (P3) node[fill=black,inner sep=1,label=right:{$a+v_n$}]{} -- (P4) -- (P5) -- (P6) -- (P7)
    -- (Est) node[circle,fill=white,draw=black,inner sep=1.5,label=right:{$a+u$}]{}
    -- (SE) -- (Sud) -- (Q1) -- (Q2) -- (Q3) -- (Q4) -- (Q5) -- (Q6) -- (Q7) -- 
    (Vest) -- (NV) -- cycle;
   \draw[dotted] (Est) -- (1,-1) -- (Sud) (Vest) -- (-1,1) -- (Nord);

    \draw[dotted, name path=rayFromA] (a) -- ($6.5*(P3)$);
\begin{scope}[shift={(-1.8, 2.2)}]
    \coordinate (Nord) at (0,1) {};
    \coordinate (Sud) at (0,-1) {};
    \coordinate (Est) at (1,0) {};
    \coordinate (Vest) at (-1,0) {};

    \node[circle,fill=black!70!white,draw=black,inner sep=1.2,label=left:{$b$}] (b) at (0,0) {};

    \coordinate (SE) at ({cos(-45)},{sin(-45)}) {};
    \coordinate (NV) at ({cos(135)},{sin(135)}) {};

    \coordinate (P1) at ({cos(60)}, {sin(60)})  {};
    \coordinate (P2) at ({cos(30)}, {sin(30)}) {};
    \coordinate (P3) at ({cos(20)}, {sin(20)}) {};
    \coordinate (P4) at ({cos(15)}, {sin(15)}) {};
    \coordinate (P5) at ({cos(10)}, {sin(10)}) {};
    \coordinate (P6) at ({cos(5)}, {sin(5)}) {};
    \coordinate (P7) at ({cos(2)}, {sin(2)}) {};

    \foreach \x in {1,...,7} {
        \coordinate (Q\x) at ($-1*(P\x)$);
    }

    \draw (Nord) -- (P1)
    -- (P2) -- (P3) -- (P4) -- (P5) -- (P6) -- (P7)
    -- (Est) node[circle,fill=white,draw=black,inner sep=1.5,label=below right:{$b+u$}]{}
    -- (SE) -- (Sud) -- (Q1) -- (Q2) -- (Q3) -- (Q4) -- (Q5) -- (Q6) -- (Q7) -- 
    (Vest) -- (NV) -- cycle;

    \draw[dashed, name path=lowerRayFromB] ({5*cos(-45)},{5*sin(-45)}) -- (b);
    \draw[dashed, name path=upperRayFromB] (b) -- (8,0);
    \draw[dotted] (Est) -- (1,-1) -- (Sud) (Vest) -- (-1,1) -- (Nord);

    \path [name intersections={of=rayFromA and lowerRayFromB,by=X}]; 
    \path [name intersections={of=rayFromA and upperRayFromB,by=Y}]; 
    \node [circle,fill=black!50!white,inner sep=1.4,label=below:{$x'$}] at (X) {};
    \node [circle,fill=black!50!white,inner sep=1.4,label=below:{$y'$}] at (Y) {};

\end{scope}

    \end{tikzpicture}
    
    \caption{How the elementary convex set $C_\pi$, from the proof of Proposition~\ref{prop:nonPolyhGauge2Dim}, with respect to $a$ and $b$ can be constructed.
    It depends on a face $F$ of $B_\gamma$, which may be either the singleton $\{u\}$ or a segment with $u$ as one endpoint (dotted case), extending downward.
    The set $C_\pi$ is the single point $\{y'\}$ when $u$ is exposed, and the line segment $[x',y']$ otherwise}
    \label{fig:ExampleProof2dim}
\end{figure}
    
    We will show that $b+\lambda_n u$ is contained in $\EH_{\gamma}(\{a,b\})$ for every $n$.
    Since $\lambda_n\to\infty$ as $n\to\infty$, this implies that $\EH_{\gamma}(\{a,b\})$ is unbounded because it will contain all the points of a ray emanating from $b$ which is parallel to $u$.

    Indeed, let $C_\pi$ an elementary convex set containing $b+\lambda_n u$ in its relative interior.
    In particular, \[C_\pi=\left(a+\RR^+(b-a+\lambda_n u)\right)\cap (b+\RR^+F)\] for an exposed face $F$ of $B_\gamma$ containing $u$.
    We have already shown that $F$ must contain elements with non-positive $x_2$-coordinates, which gives \[b+\RR^+F\subset\{x\in\RR^2:x_2\leq b_2\}.\]
    On the other hand,
    \[a+\mu(b-a+\lambda_n u)\in\{x\in\RR^2:x_2>b_2\}\ \text{ for every }\mu>1\]
    which ensures that these points do not belong to $b+\RR^+F$.
    As a consequence, $C_\pi\subset[a,b+\lambda_n u]$, which is bounded, ensuring that $b+\lambda_n u\in\EH_\gamma(\{a,b\})$.

    A situation as above can be observed in Figure~\ref{fig:ExampleProof2dim}, where $y'=b+\lambda_n u$.
    In that figure, $C_\pi=\{y'\}$ but there may be situations where $u$ is not exposed, e.g. when in Figure~\ref{fig:nonpolyhNotStrictLocalConvexSet} the points $\pm(1/\sqrt{2},-1/\sqrt{2})$ are replaced by $\pm(1,-1)$, as shown in dotted lines in Figure~\ref{fig:ExampleProof2dim}.
    In such cases, $C_\pi$ would be the segment $[x',y']$.
\end{proof}

\begin{corollary}
    A norm on $\RR^2$ is uniformly robust if and only if it is polyhedral.
\end{corollary}

\section{The contamination locus} \label{sec:CL}

Section~\ref{sec:unifRobust} focused on the contamination of a set $D$ to understand the robustness of Fermat--Weber points.
We note that this provides a similar approach to studying breakdown points of finite samples as discussed in \cite[Chapter~11]{Huber1981}.
When $\gamma$ is a norm, we saw that $\EH_\gamma(D)$ contains all possible contaminated Fermat--Weber points when the contamination is below the breakdown point.
The following example shows that $\EH_\gamma(D)$ could be too large in Proposition~\ref{prop:elemSet} in the asymmetric case.

\begin{example} \label{eg:EHlargeQuantile}
    Let $D\subset\RR$ be consisting of $n$ points $d_1<\dots<d_n$ and $C\subset\RR$ any finite (positively) weighted set such that ${\bar{w}}_C<\frac{b}{1-b} w_D$ for some $b\in]0,1/2[$.
    We will use $(C,{\bar{w}})$ to obtain a $\tau$-contamination of $(D,w)$ with $\tau={\bar{w}}_C/({\bar{w}}_C+w_D)<b$, where $b$ is the breakdown point of the $b$-quantile.
    We remind that the corresponding gauge is $\gamma(x)=\max((1-b)x,-bx)$.
    
    In other words, we analyze the case of $b$-quantiles for $(A,w')=(D,w)+(C,{\bar{w}})$, which are points $x\in\RR$ such that \[\sum_{a\in A:\,a\leq x}w'_a\geq bw'_A\quad\text{and}\quad\sum_{a\in A:\,a\geq x}w'_a\geq (1-b)w'_A.\]
    
    Let $k\in\{1,\dots,n\}$ such that \[\sum_{i>k} w_{d_i}\leq \frac{1-2b}{1-b}w_D\quad\text{and}\quad\sum_{i\geq k} w_{d_i}>\frac{1-2b}{1-b}w_D.\]
    Then the $b$-quantile of $A$ belongs to $[d_1,d_k]$.
    Indeed, let $x>d_k$.
    Then the weight of the points in $A$ greater or equal to $x$ is \[\sum_{a\in A:\,a\geq x}w'_a 
    \leq {\bar{w}}_C+\sum_{i>k} w_{d_k}\leq {\bar{w}}_C+\frac{1-2b}{1-b}w_D<(1-b)({\bar{w}}_C+w_D)=(1-b)w'_A,\] where the last inequality comes from ${\bar{w}}_C<\frac{b}{1-b}w_D$ by multiplication with $b$ followed by the addition of $\frac{1-2b}{1-b}w_D$.
    The last relation shows that $x$ cannot be a $b$-quantile for $A$.

    This is a counterexample to the extension of Theorem~\ref{th:EHnorms} to asymmetric gauges when $[d_1,d_k]$ is a strict subset of $\EH_\gamma(D)=[d_1,d_n]$, which happens as soon as ${w_{d_n}\leq\frac{1-2b}{1-b}w_D}$.
\end{example}

Consequently, Theorem~\ref{th:EHnorms} cannot be extended to the asymmetric case.
Example~\ref{eg:EHlargeQuantile} points out that the weights on $D$ play a role in the study of the contamination of $(D,w)$ in the absence of symmetry.
The aforementioned example motivates the study of the set

\[\CL_{\gamma}(D,w):=\bigcup \Big\{ \FW_\gamma(A,w') :
(A,w')=(D,w)+(C,{\bar{w}}),\,\sigma {\bar{w}}_C<w_D
\Big\}\]
which we call the \emph{contamination locus} of the weighted set $(D,w)$.
This is similar to the set of efficient points studied by Durier~\cite{Durier:90}, although here we fix a large part of the sample.

\begin{lemma} \label{lem:one-pointCL}
    Let $x$ be a Fermat--Weber point of a sample $(D,w)+(C,{\bar{w}})$ with $\sigma\cdot {\bar{w}}_C<w_D$.
    Then there exists $e\in\RR^d$ and $w''_e<w_D/\sigma$ such that $x$ is a Fermat--Weber point of $(D,w)+(e,w''_e)$.
\end{lemma}

\begin{proof}
    From the first-order optimality condition for convex minimization, we have \[\0\in\sum_{c\in C}{\bar{w}}_c\partial\gamma(x-c)+\sum_{d\in D}{\bar{w}}_c\partial\gamma({x-d}).\]
    Hence, there exists \[q\in\sum_{c\in C}{\bar{w}}_c\partial\gamma({x-c})\ \text{ such that }-q\in\sum_{d\in D}w_d\partial\gamma({x-d}).\]

    If $q=\0$, then we can choose $e=x$ and $w''_e$ any value smaller than $w_D/\sigma$.
    Otherwise, let \[w''_e=\gamma^\circ(q)\leq {\bar{w}}_C<w_D/\sigma\]
    and set $e=x-u$ for some $u\in\partial\gamma^\circ(q)$.
    Since $x-e=u\in\partial\gamma^\circ(q)$ and $q/w''_e\in\partial B_{\gamma^\circ}$, we also have $q/w''_e\in\partial\gamma(x-e)$.
    Thus \[\0=q+(-q)\in w''_e\partial\gamma(x-e)+\sum_{d\in D}w_d\partial\gamma(x-d).\qedhere\]
\end{proof}

In other words, Lemma~\ref{lem:one-pointCL} says that
\[\CL_{\gamma}(D,w)=\bigcup \Big\{ \FW_\gamma(A,w') :
(A,w')=(D,w)+(e,w''_e),\,e\in\RR^d,\,0<w''_e<w_D/\sigma
\Big\}.\]
We use this description to obtain more information about the contamination locus.

\begin{proposition} \label{prop:propertiesCL}
    The contamination locus $\CL_\gamma(D,w)$ is a connected set such that $\FW_\gamma(D,w)\subset\CL_\gamma(D,w)\subset\EH_\gamma(D)$.
    What is more, the contamination locus is a union of elementary convex sets for $D$.
\end{proposition}

\begin{proof}
    The inclusion $\FW_\gamma(D,w)\subset\CL_\gamma(D,w)$ comes from setting $e$ to be any point of $\FW_\gamma(D,w)$ and the inclusion $\CL_\gamma(D,w)\subset\EH_\gamma(D)$ comes from Proposition~\ref{prop:elemSet}.

    To prove that the contamination locus is a union of elementary convex sets for $D$, we select $x\in\CL_\gamma(D,w)$ and $C_\pi$ the smallest elementary convex set containing $x$, which must be bounded since $x\in\EH_\gamma(D)$. Note also that $x$ must lie in the relative interior of $C_\pi$, since the relative interiors of elementary convex sets partition $\RR^d$.
    According to Lemma~\ref{lem:one-pointCL}, there exists $e\in\RR^d$ and $w''_e<w_D/\sigma$ such that \[\0\in\partial g(x)+w''_e\partial\gamma({x-e}),\]
    where $g(x)=\sum_{d\in D}w_d\gamma(x-d)$ as defined in section~\ref{sec:upperBound}.
    
    We know that $\partial g$ is constant over the relative interior of $C_\pi$ and is a larger set at points in the relative boundary of $C_\pi$.
    Therefore, for every $y\in C_{\pi}$ we have \[\partial g(y)+w''_e\partial\gamma(y-(e+y-x))\supseteq\partial g(x)+w''_e\partial\gamma(x-e)\ni\0.\]
    Consequently, adding $(e+y-x,w''_e)$ to $(D,w)$ makes $y$ a Fermat--Weber point.
    In other words, $y\in\CL_\gamma(D,w)$ for every $y\in C_{\pi}$.
    This proves the last statement of the proposition.

    We prove the connectedness of $\CL_\gamma(D,w)$ using set-valued analysis.
    Consider the continuous function $f:\RR^d\times\RR^d\times]0,w_D/\sigma[\to\RR$ defined by \[f(x,e,w''_e)=\sum_{d\in D}w_d\gamma({x-d})+w''_e\gamma({x-e}),\] which is a parametric family of objective functions for the Fermat--Weber problem of contaminations of $(D,w)$.
    Fixing $a^\star\in\FW_\gamma(D,w)$ and seeing $(D,w)+(e,w''_e)$ as the corruption of $(D,w)+(a^\star,w''_e)$, we know that $f(\cdot,e,w''_e)$ attains its minimum on the set
    \[F(e,w''_e)=\left\{x\in\RR^d:\gamma(a^\star-x)\leq\frac{w_D}{w_D-\sigma w''_e}\sigma(1+\sigma)M\right\}\]
    from \eqref{eq:quantitativeRobustness} where $M=\max_{d\in D}\gamma(a^\star-d)$.
    This is the ball $a^\star-r(w''_e)B_\gamma$, where $r$ is a continuous and increasing function on $]0,w_D/\sigma[$.
    In particular, $F$ is a continuous set-valued function with compact values.
    Thus, the set-valued function \[\Theta(e,w''_e):=\argmin_{x\in F(e,w''_e)}f(x,e,w''_e)\] is upper semi-continuous from Berge's maximum theorem \cite[p.~116]{Berge:topological}.
    We note that its image is \[\Theta\left(\RR^d\times]0,w_D/\sigma[\right)=\CL_\gamma(D,w).\]
    Since $\Theta$ is upper semi-continuous and compact-valued, \cite[Theorem~3.1]{Hiriart-Urruty:85} entails that the image of connected sets are also connected.
    This shows that $\CL_\gamma(D,w)$ is connected.
\end{proof}

\begin{example} \label{eg:quantileCL}
    Let us compute $\CL_b(D,w):=\CL_{\gamma_b}(D,w)$ for a weighted set $D=\{{d_1,\dots,d_n}\}\subset\RR$, where $d_1<\dots<d_n$, and the gauge $\gamma_b(x)=\max(({1-b})x,-bx)$ defining the $b$-quantile.
    In other words, $\CL_b(D,w)$ is the contamination locus of the $b$-quantile of $(D,w)$.
    For simplicity, we assume $b\in]0,1/2[$.

    Let $k\in[n]$ such that \[\sum_{i>k}w_{d_i}\leq\frac{1-2b}{1-b}w_D\ \text{ and }\ \sum_{i\geq k}w_{d_i}>\frac{1-2b}{1-b}w_D.\]
    We showed in Example~\ref{eg:EHlargeQuantile} that $\CL_b(D,w)\subset[d_1,d_k]$.
    In fact, we will show that we have equality.
    Since the contamination locus is connected, it is sufficient to show that $d_1$ and $d_k$ belong to $\CL_b(D,w)$

    We can show that $d_1\in\CL_b(D,w)$ by considering \[e<d_1\ \text{ and }\ w''_e\geq \frac{b}{1-b}w_D-\frac{w_{d_1}}{1-b}.\]
    Similarly, to show $d_k\in\CL_b(D,w)$, we have to set \[e>d_k\ \text{ and }\ w''_e\geq\frac{b}{1-b}w_D-\frac{1}{b}\sum_{i\geq k}w_{d_i}.\]
    In either case, $d_1$ or $d_k$ is a $b$-quantile for $(D,w)+(e,w''_e)$.
\end{example}

    \begin{figure}
    \centering

    \begin{tabular}{p{0.4\textwidth}p{0.4\textwidth}}
        \resizebox{0.3\textwidth}{!}{
            \begin{tikzpicture}[scale=1.3,x={(0.866,-0.5)},y={(0,1)},z={(-0.866,-0.5)}]
            \draw[semithick] (-1,-1) -- (1,0) -- (0,1) -- cycle;
    
            \draw[very thick,->] (0,0) -- (0.5,0.5);
            \draw[very thick,->] (0,0) -- (-0.5,0);
            \draw[very thick,->] (0,0) -- (0,-0.5);
    
            \node[circle,fill,inner sep=1pt] at (0,0) {};
        
            \node at (0,-0.7) [circle, draw= white, fill = white, semithick, inner sep = 1.5pt] {\footnotesize a};
            \node at (0.7,0.7) [circle, draw= white, fill = white, semithick, inner sep = 1.5pt] {};
            \node at (-0.7,0) [circle, draw= white, fill = white, semithick, inner sep = 1.5pt] {};
        
        \end{tikzpicture}
        } & \resizebox{0.3\textwidth}{!}{
            \begin{tikzpicture}[scale=1.3,x={(0.866,-0.5)},y={(0,1)},z={(-0.866,-0.5)}]
        \filldraw[fill=black!20!white, draw=black!50!white, semithick] (0,1) -- (1,1) -- (1,0) -- (0,-1) -- (-1,-1) -- (-1,0) -- cycle;

        \draw[ultra thick] (0,0) -- (1,0);
        \draw[ultra thick] (0,0) -- (0,1);
        \draw[ultra thick] (0,0) -- (-1,-1);

        \node at (0,0) [circle, draw= black, fill = black, semithick, inner sep = 1.3pt] {};
        
        \node at (0,1) [circle, draw= black, fill = white, semithick, inner sep = 1.5pt, label={[black]90:\footnotesize$1$}] {};
        \node at (1,1) [circle, draw= black, fill = white, semithick, inner sep = 1.5pt, label={[black]10:\footnotesize$3$}] {};
        \node at (1,0) [circle, draw= black, fill = white, semithick, inner sep = 1.5pt, label={[black]-10:\footnotesize$1$}] {};
        \node at (0,-1) [circle, draw= black, fill = white, semithick, inner sep = 1.5pt, label={[black]below:\footnotesize$3$}] {};
        \node at (-1,-1) [circle, draw= black, fill = white, semithick, inner sep = 1.5pt, label={[black]260:\footnotesize$1$}] {};
        \node at (-1,0) [circle, draw= black, fill = white, semithick, inner sep = 1.5pt, label={[black]170:\footnotesize$3$}] {};
    \end{tikzpicture}
        } \\
        \small (a) Equilateral triangle with skewness directions & \small (b) Six points (white) with adjoining weights, their elementary hull (grey), and contamination locus (black)
    \end{tabular}
    
        \caption{The unit ball of the gauge distance is displayed in (a), whereas (b) contains the values described in Example~\ref{eg:CL-simplex}}
        \label{fig:CL}
    \end{figure}

\begin{example} \label{eg:CL-simplex}
    Consider a gauge $\gamma$ whose unit ball is an equilateral triangle $\triangle$, as in Figure~\ref{fig:CL}~(a).
    Let $D$ be the set consisting of the vertices of $\triangle$, each receiving weight $1$, and their opposites, each having weight $3$; they are coloured white in Figure~\ref{fig:CL}~(b).
    The origin is the unique Fermat--Weber point of $D$.
    
    The elementary hull $\EH_\gamma(D)$ coincides with the convex hull of $D$ and is coloured grey in Figure~\ref{fig:CL}~(b).
    The contamination locus $\CL_\gamma(D,w)$ is a strict subset of $\EH_\gamma(D)$ and consists of the union of three segments with one endpoint at the origin.
    The contamination locus is drawn with black in Figure~\ref{fig:CL}~(b) and includes also the three points of $D$ that are endpoints of the aforementioned segments.
    Computational details are given in Appendix~\ref{sec:CL_trop_med}.
\end{example}

\section{Majority rules} \label{sec:maj-rule}

Majority rules in location theory constitute a rich topic: initiated for the Fermat--Weber problem by Witzgall~\cite{Witzgall:TechnicalReport:1964} (see also~\cite{Plastria2022} for a recent view), with possibly negative weights in \cite{Plastria-majFW}, somewhat extended for a majority region \cite{Goldman+Witzgall:1970}, in problems with mixed distances \cite{BLM:2009, Plastria2011,Plastria:2009}, and for problems with several interconnected facilities to be located \cite{Fliege1997, Lefebvre1991, Plastria2024}.

This section will briefly discuss the relation between breakdown points and majority rules in location theory.
Such connections were already noted in \cite[\textsection 5]{Griffith+Chun+Kim:2022}, where Witzgall's majority rule is explained as follows:
    \begin{quote}
        In other words, for the global spatial median case, if the weight at any location is more than half of the total sum of all $n$ weights, then the spatial median collocates with that demand point.
        In this context, weights function in a fashion similar to frequencies of observations in classical statistics.
        [Witzgall's majority rule] relates to the breakdown point of a median, which is 50\%.
    \end{quote}

In fact, the above observation extends also for the asymmetric case.
The generalization was previously obtained by Plastria~\cite{Plastria-majFW} with more elementary approaches, but we give below a different proof based on our robustness study.

\begin{theorem}[Plastria]\label{th:majority}
    If $\sigma$ is the skewness of $\gamma$ and there is $u\in A$ such that $w_u\geq\sigma\cdot w_{A\setminus\{u\}}$, then $u\in\FW_\gamma(A,w)$.
    What is more, if $w_u>\sigma\cdot w_{A\setminus\{u\}}$, then we actually have $\FW_\gamma(A,w)=\{u\}$.
\end{theorem}

\begin{proof}
    Firstly, assume that $w_u>\sigma\cdot w_{A\setminus\{u\}}$ and we apply Proposition~\ref{prop:elemSet} for $C=A\setminus\{u\}$.
    Consequently, all Fermat--Weber points belong to $\EH_\gamma(\{u\})=\{u\}$ which shows that $\FW_\gamma(A,w)=\{u\}$ when $w_u>\sigma\cdot w_{A\setminus\{u\}}$.

    In other words, we showed that $\sum_{a\in A}w_a\gamma(u-a)<\sum_{a\in A}w_a\gamma(x-a)$ for every $x\in\RR^d\setminus\{u\}$ provided that $w_u>\sigma\cdot w_{A\setminus\{u\}}$.
    Passing to the limit $w_u\to\sigma\cdot w_{A\setminus\{u\}}$, we obtain $\sum_{a\in A}w_a\gamma(u-a)\leq\sum_{a\in A}w_a\gamma(x-a)$ for every $x\in\RR^d\setminus\{u\}$ when $w_u=\sigma\cdot w_{A\setminus\{u\}}$.
    Put differently, $u$ is a Fermat--Weber point.
    This concludes the fact that $u\in\FW_\gamma(A,w)$ whenever $w_u\geq\sigma\cdot w_{A\setminus\{u\}}$.
\end{proof}

Witzgall's majority theorem addresses the symmetric case ($\sigma=1$). Specifically, it states that $a$ is a Fermat--Weber point when $w_a \geq w_{A\setminus\{a\}}$, though this point may not be unique in case of equality.
For instance, consider the median of two distinct real numbers.

In the context of consensus methods, \cite{Comaneci:2024} uses breakdown point analysis to establish other types of majority rules for phylogenetic trees. For certain properties (P), if the proportion of trees with property (P) exceeds $\sigma/(1+\sigma)$, then the Fermat--Weber point also exhibits property (P). Such properties (P) that are listed in \cite{Comaneci:2024} are tied to tropical convex hulls which are specific cases of elementary hulls. This situation arises in the case of simplicial gauges.

More specifically, Proposition~\ref{prop:elemSet} allows extending majority rules to sample properties: if every point in a weighted subsample $(B,{\bar{w}})$ possesses property (P) and ${\bar{w}}_B > \sigma \cdot (w_A-{\bar{w}}_B)$, then every point in $\FW_\gamma(A,w)$ inherits property (P).
This occurs when the set
\begin{equation} \label{eq:genericPset}
    \{x\in\RR^d:x\text{ has property (P)}\}
\end{equation}
is closed under elementary hulls.
This scenario appears in \cite[Proposition~29]{Comaneci:2024} when discussing Pareto properties of certain location statistics for phylogenetic trees.
In this case, the simplicial gauge distance ensures that the elementary hull coincides with the tropical convex hull.
Tropically convex sets, i.e. those closed under tropical convex hulls, of type \eqref{eq:genericPset} are listed in \cite[\textsection 4.3]{Comaneci-PhD}.
The aforementioned sets play a crucial role in the study of consensus methods in phylogenetics.

A challenge arises because other methods in \cite{Comaneci:2024} lack a clear connection to tropical convex hulls, as the elementary hulls involved are distinct. As a result, contamination loci need to be considered with thresholds higher than $\sigma/(1+\sigma)$.

However, elementary hulls do not generally form a closure operator~\cite[\textsection 2.1.2]{Ganter+Obiedkov:2016}.
Identifying sets closed under elementary hulls, particularly in non-simplicial cases, remains a difficult task.
A deeper understanding of contamination loci could facilitate the extension of majority rules.

\section{Final remarks} \label{sec:conclusion}

In this paper we have derived the robustness of the Fermat--Weber problem for any gauge distance.
Nevertheless, many practical situations involve other distances~\cite{Deza+Deza-distances} and it is of interest to see if similar results would hold for more general settings.
In particular, it would provide measures of robustness in the computations of medians and clustering algorithms.

Some indication in this direction exists due to the fact that the lower bound theorem~\ref{th:lowerBound} was proven without the complex machinery of convex analysis, but using only continuity of the distance and triangle inequality with equality on segments.
This latter property was recently used to define segmented metrics in \cite{Plastria2022}.
We remark that Funk distances are also common quasi-metrics which are linear along segments; see \cite[Chapter~2]{handbook-hilbert-geometry}.
Some of the results of this paper might, therefore, extend to such quasi-metrics with finite skewness.

We also conjecture that the computations of breakdown points of M-estimators reduces to the Fermat--Weber case.
Recall that an \emph{M-estimator} arises as the minimum of a function of the form
\[x\mapsto\sum_{a\in A}\rho(x-a)\]
where $A$ is the given sample and $\rho$ is a fixed function.
The Fermat--Weber case arises when $\rho$ is a gauge.

Our analysis does not extend directly, as we lack the invariance of subdifferentials under scaling or linearity over segments.
However, the breakdown occurs when some part of $A$ is arbitrarily corrupted, so we have a movement to ``infinity.''
Such asymptotic behaviour is encoded by the \emph{horizon function} $\rho^\infty$ as defined in \cite[\textsection 3.C]{VarAnal}.
We remark that this is a gauge and we conjecture that the corresponding Fermat--Weber point has the same breakdown point with our M-estimator.
We remark that this conjecture holds for the 1-dimensional case, according to the analysis in \cite[\textsection 3.2.1]{Maronna2018}.

We note that the practical identification of corrupted versus uncorrupted subsets (such as the sets $C$ and $D$ considered in section~\ref{sec:upperBound}) is closely related to the problem of \emph{outlier detection}~\cite{Rousseeuw+Leroy:1987,Rousseeuw+Hubert-anomaly-detection:2018}.
Exploring such applications is a promising direction for future work and would likely require incorporating robust measures of dispersion alongside location.
Nevertheless, we note that this connection is not straightforward: for example, certain cases of our Fermat--Weber estimators, such as M-quantiles, have been found less effective for outlier detection~\cite{chakroborty2024usemquantilesoutlierdetection}.

We saw that the uniform robustness is related to the boundedness of the elementary hull.
This is easy to check for polyhedral or locally strictly convex norms, but our approach does not extend to more general norms, e.g. cylindrical.
In any case, we give the following conjecture which holds in dimension at most two.

\begin{conjecture} \label{conj:unifRobustNorm}
    A norm is uniformly robust if and only if it is polyhedral.
\end{conjecture}

When we lose symmetry, analyzing the elementary hull is not sufficient.
Theorem~\ref{th:upperBound} and Examples~\ref{eg:quantileCL} and~\ref{eg:CL-simplex} indicate that the weights and the skewness directions play an important role; the asymmetric case is highly more complex.
Moreover, it is plausible that a similar statement to Conjecture~\ref{conj:unifRobustNorm} is false for asymmetric gauges.
It might be that $\gamma$ is uniformly robust if it is polyhedral in a neighbourhood of $-\SD_\gamma$.

\subsection*{Acknowledgements}
We thank Michael Joswig as well as an unknown referee for useful comments.

\section*{Declarations}

\subsection*{Funding}
A. Com\u{a}neci was supported by the Research Training Group ``Facets of Complexity'' (GRK 2434, project-ID 385256563).

\subsection*{Competing Interests}
The authors declare that they have no conflicts of interest.

\subsection*{Compliance with Ethical Standards}
This article does not contain any studies involving human participants  or animals performed by any of the authors.

\begin{appendices}

\section{Contamination locus for tropical medians} \label{sec:CL_trop_med}

In this section, we discuss a way of computing contamination loci for tropical medians~\cite{TropMedian}, which arise when the unit ball of the gauge is a regular simplex.
We will focus only on Example~\ref{eg:CL-simplex}, although the ideas extend naturally to other weighted sets.
However, we firstly need some helpful results holding for every gauge.

\begin{lemma} \label{lem:aCL}
    For every $a\in\CL_\gamma(D,w)$ there exists $w'_a\in]0,w_D/\sigma[$ such that $a\in\FW_\gamma((D,w)+(a,w'_a))$.
\end{lemma}

\begin{proof}
    Since $a\in\CL_\gamma(D,w)$, Lemma~\ref{lem:one-pointCL} implies the existence of $(e,w'_e)\in\RR^d\times]0,w_D/\sigma[$ such that $a\in\FW_\gamma((D,w)+(e,w'_e))$.
    In other words, using the notation $g(x)=\sum_{d\in D}w_d\gamma(x-d)$, the point $a$ is a minimum of the function $g(\cdot)+w'_e\gamma(\cdot-e)$.
    We will prove that $a$ is also a minimum of $h(\cdot)=g(\cdot)+w'_e\gamma(\cdot-a)$.

    Using the triangle inequality and the optimality of $a$, we obtain
    \[\begin{split}
        h(x)=g(x)+w'_e\gamma(x-a) &\geq g(x)+w'_e\gamma(x-e)-w'_e\gamma(a-e) \\
        & \geq g(a)+w'_e\gamma(a-e)-w'_e\gamma(a-e)\\
        & = g(a) = h(a).
    \end{split}
    \]
    This shows that $a$ is a minimum of $h$.
    Equivalently, $a\in\FW_\gamma((D,w)+(a,w'_e))$.
\end{proof}

\begin{corollary} \label{cor:aCL}
    If $a\in\CL_\gamma(D,w)$, then $\{a\}=\FW_\gamma((D,w)+(a,w_D/\sigma))$.
\end{corollary}

\begin{proof}
    According to Lemma~\ref{lem:aCL}, $a$ is a minimum of $h(\cdot):=\sum_{d\in D}w_d\gamma(\cdot-d)+w'_a\gamma({\cdot-a})$ for some $w'_a\in]0,w_D/\sigma[$.
    We know that $a$ is the unique minimum of $({w_D/\sigma-w'_a})\gamma(\cdot-a)$, so $a$ is also the unique minimum of $h(\cdot)+({w_D/\sigma-w'_a})\gamma(\cdot-a)$.
    The last statement is equivalent to $\{a\}=\FW_\gamma((D,w)+(a,w_D/\sigma))$.
\end{proof}

\begin{figure}
    \centering
    \begin{tikzpicture}[scale=1.3,x={(0.866,-0.5)},y={(0,1)},z={(-0.866,-0.5)}]



        \draw (0,1) -- (-1.5,-0.5);
        \draw (0,1) -- (1.5,1);

        \draw (0,-1) -- (0,1.5);

        \draw (-1,-1) -- (-1,1);
        \draw (-1,-1) -- (0.5,-1);

        \draw (1,0) -- (1,2);
        \draw (1,0) -- (-0.5,-1.5);

        \draw (-1,0) -- (1.5,0);

        \draw (1,1) -- (-1.5,-1.5);
        
        \node at (0,1) [circle, draw= black, fill = white, semithick, inner sep = 1.5pt, label={[black]170:$1$}] {};
        \node at (1,1) [circle, draw= black, fill = white, semithick, inner sep = 1.5pt, label={[black]10:$3$}] {};
        \node at (1,0) [circle, draw= black, fill = white, semithick, inner sep = 1.5pt, label={[black]below:$1$}] {};
        \node at (0,-1) [circle, draw= black, fill = white, semithick, inner sep = 1.5pt, label={[black]below:$3$}] {};
        \node at (-1,-1) [circle, draw= black, fill = white, semithick, inner sep = 1.5pt, label={[black]below:$1$}] {};
        \node at (-1,0) [circle, draw= black, fill = white, semithick, inner sep = 1.5pt, label={[black]170:$3$}] {};

        \node at (0.25,0.5) [circle, draw= white, fill = black, semithick, inner sep = 1.5pt] {};

        \draw[dashed] (0.25,0.5) -- (1.5,1.75);
        \draw[dashed] (0.25,0.5) -- (-1.5,0.5);
        \draw[dashed] (0.25,0.5) -- (0.25,-1.25);
        
    \end{tikzpicture}
    \caption{Subdivision of $\RR^2$ in elementary convex sets for $D$ (white points) with adjoining weights and a point (black) in the interior of $\EH_\gamma(D)$}
    \label{fig:ec-trop-median}
\end{figure}

The tropical median was defined in \cite{TropMedian} as the Fermat--Weber point under a regular simplicial gauge.
An explicit formula is given on the quotient vector space $\torus{d+1}$ by $\gamma_\triangle(x)=\sum_{i=0}^d x_i-(d+1)\min_{0\leq j\leq d} x_j$.
Indeed, one can check that the unit ball is the image of the regular $d$-simplex $\Delta_{d}=\conv\{e_0,\dots,e_d\}$ in $\torus{d+1}$, where $e_i$ represents a standard unit vector in $\RR^{d+1}$.
The unit ball for $d=2$ is displayed in Figure~\ref{fig:CL}~(a).
The skewness of $\gamma_\triangle$ is $\sigma = d$ and its dual gauge has the expression $\gamma_\triangle^\circ(x)=\gamma_\triangle(-x)/(d+1)$.

For computing tropical medians, particularly for identifying the contamination locus, we use a balancing condition given by \cite[Theorem~15]{TropMedian}.
Instead of mentioning it in full generality, we will specify how to use it in Example~\ref{eg:CL-simplex}.

To check if a point $m$ is a tropical median for a weighted set $(A,w)$ in $\RR^2\simeq\torus{3}$ we need to do the following:
\begin{enumerate}
    \item Consider the regions induced by the cones over the facets of $B_{\gamma_\triangle^\circ}=-3\Delta_2$ based at $m$.
    For example, if $m$ is the black point from Figure~\ref{fig:ec-trop-median}, there are three regions delimited by dashed lines: upper, lower-left, and lower-right.
    \item In each of the regions above, the sum of the weights of points contained in it (including the boundary) must be at least $w_A/3$.
    \item The sum of the weights of points contained in the union of \emph{two} regions (including the boundary) must be at least $2\cdot w_A/3$.
    Complementarily, the total weight in the \emph{interior} of each region as above must be at most $w_A/3$.
\end{enumerate}

Using these rules, if we choose $a$ as the black point from Figure~\ref{fig:ec-trop-median}, then $a$ is not a tropical median for $(A,w)=(D,w)+(a,6)$.
By Corollary~\ref{cor:aCL}, this yields $a\notin\CL_\gamma(D,w)$.
We note that the same argument works for each interior point of a bounded and full-dimensional elementary convex set for $D$, due to the symmetry of $D$.
Indeed, we note that the total weight of the points from the interior of lower-left region is $7$ which is greater than $w_A/3=(12+6)/3=6$.
Rule~(3) from above indicates that $a\notin\FW_\gamma((D,w)+(a,6))$.
A similar argument holds for the case when $a$ is on the open segment between the origin and a point of weight $3$.

On the other hand, if $a$ belongs to the open segment between the origin and a point of weight $1$ and set $w_a=5$, then the interior of the three regions from (1) will have weights $1$, $4$, and $4$.
All of them are smaller than $w_A/3=17/3$, so (3) applies.
But if we take the boundary into consideration, the weights are $6$, $9$, and $9$; the point $a$ belongs to all regions, so we added its weight $w_a=5$.
All of these values are greater than $17/3$, so we also satisfy (2).
All in all, $a$ will be a Fermat--Weber point for $(D,w)+(a,5)$, showing that the contamination locus from Figure~\ref{fig:CL} was correctly computed.

\section{Elementary convex sets} \label{sec:ElemConvSets}

In this section, we will prove a few basic results about elementary convex sets that we could not find in the literature.
We will use the notation from section~\ref{sec:unifRobust}.

Firstly, we will need a simple lemma about the sets $N(p)$.

\begin{lemma} \label{lemma:inters_Np_with_Nq}
    For every $p,q\in B_{\gamma^\circ}$ and every $\lambda\in]0,1[$ we have $N(p)\cap N(q)=N({\lambda p+(1-\lambda)q})$.
\end{lemma}

\begin{proof}
    We will prove the result for $p,q\in\partial B_{\gamma^\circ}$ as for other cases we trivially have $N(p)\cap N(q)=\{\0\}=N(\lambda p+(1-\lambda)q)$ as we defined $N(r)=\{\0\}$ for $\gamma(r)<1$.

    If $N(p)\cap N(q)=\{\0\}$, then $\partial\gamma^\circ(p)\cap\partial\gamma^\circ(q)=\emptyset$ since $N(p)=\RR^+\partial\gamma^\circ(p)$.
    Consequently, the open segment $]p,q[$ belongs to the interior of $B_{\gamma^\circ}$.
    Therefore, $N({\lambda p+(1-\lambda)q})=\{\0\}$.

    Now, consider the case when $N(p)\cap N(q)\neq\{\0\}$.
    Since $N(p)=\RR^+\partial\gamma^\circ(p)$ for every $p\in\partial B_{\gamma^\circ}$, then the conclusion is equivalent to $\partial\gamma^\circ(p)\cap\partial\gamma^\circ(q)=\partial\gamma^\circ(r)$ for every $r$ on the open segment $]p,q[$.
    But this is a direct consequence of \cite[Lemma~8]{Plastria-pasting-gauges1}.
\end{proof}

For an elementary convex set $S$ with respect to a set $A$ and gauge $\gamma$ we say that a generating set $\pi=(p_a)_{a\in A}$ is \emph{minimal} if there is no generating set $\pi'=(p_a')_{a\in A}$ such that $N(p_b')\subsetneq N(p_b)$ for some $b\in A$.

Consider for example the elementary convex set $C_\pi=L$ of Figure \ref{fig:TriangleEH}. Any family $\pi=(p_2,p_3,p_c,p_2,p_e)$
with $(p_c,p_e)\in]p_2,p_3[^2$ is minimal, while the generating family 
where $(p_c,p_e)=(p_3,p_2)$ is not.

\begin{remark}
    If $\pi=(p_a)_{a\in A}$ and $\pi'=(p_a')_{a\in A}$ induce the same elementary convex set, $S$, then $(\pi+\pi')/2:=\left((p_a+p_a')/2\right)_{a\in A}$ also induces~$S$.
    Indeed, Lemma~\ref{lemma:inters_Np_with_Nq} implies that $N\left((p_a+p_a')/2\right)=N(p_a)\cap N(p_a')$ so
    \[\begin{split}
    C_{(\pi+\pi')/2}&=\bigcap_{a\in A}(a+N\left((p_a+p_a')/2\right))\\
    &=\bigcap_{a\in A}(a+N(p_a)\cap N(p_a'))\\
    &=\bigcap_{a\in A}[(a+N(p_a))\cap(a+N(p_a'))]\\
    &=\bigcap_{a\in A}(a+N(p_a))\cap\bigcap_{a\in A}(a+N(p_a'))\\
    &=C_\pi\cap C_{\pi'}=S.
    \end{split}\]
\end{remark}

\begin{lemma}
    If $\pi$ is a minimal generating set for an elementary convex set $S$ and $C_{\pi''}=S$, then $N(p_a)\subset N(p_a'')$.
\end{lemma}

\begin{proof}
    If it were not the case, then there exists $b\in A$ such that $N(p_b)\not\subset N(p_b'')$.
    Accordingly, $\pi':=(\pi+\pi'')/2$ is also a generating set with $N(p_b')=N(p_b)\cap N(p_b'')\subsetneq N(p_b)$.
    The last strict inclusion contradicts the minimality of $\pi$.
\end{proof}

\begin{proposition}
    Let $S$ be an elementary convex set and $x$ a point in the relative interior of $S$. If $p_a$ lies in the relative interior of $\partial\gamma(x-a)$  for any $a\in A$, then $\pi=(p_a)_{a\in A}$ is a minimal generating set of $S$.
\end{proposition}

\begin{proof}
    By (\ref{eq:xincone}) the condition $p_a\in\partial\gamma(x-a)$ for every $a\in A$ implies that $x\in C_\pi$.
    Since relative interiors of elementary convex sets are disjoint, the last relation implies $S\subset C_\pi$.
    
    Now, consider $\rho=(q_a)_{a\in A}$ such that $S=C_{\rho}$.
    In particular, $q_a\in\partial\gamma(x-a)$.
    Since $p_a$ belongs to the relative interior of $\partial\gamma(x-a)$ and the association $p\mapsto\partial\gamma^\circ(p)$ induces an antitone isomorphism between the lattice of exposed faces of $B_\gamma^\circ$ and the lattice of exposed faces of $B_\gamma$~\cite[Prop.~4.7]{Weis:2012}, then we have $\partial\gamma^\circ(p_a)\subset\partial\gamma^\circ(q_a)$.
    This implies
    \[C_\pi=\bigcap_{a\in A}(a+N(p_a))\subset\bigcap_{a\in A}(a+N(q_a))=C_{\rho}=S.\]
    Since we have already proven that $S\subset C_\pi$, then we must have $S=C_{\pi}$.
    That is, $\pi$ is a generating set of $S$.
    What is more, $\pi$ must be minimal, since we showed that $N(p_a)\subset N(q_a)$ for an arbitraty generating family~$\rho$.
\end{proof}

\end{appendices}


\bibliography{sn-bibliography}

\end{document}